\documentclass{amsart}

\usepackage[all]{xy}

\newfont{\cyr}{wncyr10}

\newcommand{\LL}{\mathbf{L}}

\newcommand{\HH}{\mathrm{H}}
\newcommand{\BB}{\mathrm{B}}
\newcommand{\ZZ}{\mathrm{Z}}

\theoremstyle{plain}
\newtheorem{thm}{Theorem}[section]
\newtheorem{prop}[thm]{Proposition}
\newtheorem{lemma}[thm]{Lemma}

\newtheorem{cor}[thm]{Corollary}

\newtheorem{question}[thm]{Question}

\theoremstyle{definition}
\newtheorem{dfn}[thm]{Definition}

\newtheorem{hypo}{Hypothesis}

\theoremstyle{remark}
\newtheorem{rem}[thm]{Remark}

\begin{document}

\title[Finiteness Theorems]{Finiteness Theorems for Deformations of Complexes}
\author[F. Bleher]{Frauke M. Bleher*}
\thanks{*Supported by 
NSA Grant H98230-06-1-0021 and NSF Grant  DMS0651332.}
\address{F.B.: Department of Mathematics\\University of Iowa\\
Iowa City, IA 52242-1419, U.S.A.}
\email{fbleher@math.uiowa.edu}

\author[T. Chinburg]{Ted Chinburg**}
\thanks{**Supported  by NSF Grants  DMS0500106
and DMS0801030. }
\address{T.C.: Department of Mathematics\\University of
Pennsylvania\\Philadelphia, PA
19104-6395, U.S.A.}
\email{ted@math.upenn.edu}

\date{September 10, 2010}

\begin{abstract}
We consider  deformations of bounded complexes of modules for a profinite group $G$ over a field of positive characteristic.  
We prove a  finiteness theorem which provides some sufficient conditions for the versal deformation of such a complex to be 
represented by a complex of $G$-modules that is strictly perfect over the associated versal deformation ring.
\end{abstract}
\maketitle

\setcounter{tocdepth}{1}
\tableofcontents


\section{Introduction}
\label{s:intro}
\setcounter{equation}{0}

The object of this paper is to determine the versal deformation rings 
and versal deformations of
bounded complexes of modules for a profinite group.  Our main result shows
that under certain hypotheses, the versal deformation may be represented
by a bounded complex of modules which are finitely generated 
over the versal deformation ring.  This is evidence for the idea
that complexes of modules which arise from arithmetic should
have versal deformations with this property.  

Suppose that $k$ is a  field of characteristic $p>0$ and that $G$ is a
profinite group.  In \cite{maz1}, Mazur developed a deformation theory of
finite dimensional representations of $G$ over $k$ using work of 
Schlessinger in \cite{Sch}.  
In \cite{comptes,bcderived}, we generalized Mazur's deformation theory by
considering, instead of $k$-representations of $G$, objects $V^\bullet$ in the derived
category $D^-(k[[G]])$ of bounded above complexes of pseudocompact modules 
over the completed
group algebra $k[[G]]$ of $G$ over $k$.  The case of $k$-representations
amounts to studying complexes that have exactly one non-zero cohomology
group.   

  As in \cite{bcderived}, we will assume $V^\bullet$
is bounded and has finite dimensional cohomology groups, and that $G$ has a 
certain finiteness property so as to be able to apply Schlessinger's work. 
The calculation of the versal deformation ring 
$R(G,V^\bullet)$ would in principle require an infinite number of
first order obstruction calculations, as discussed in \cite{bcbigobstructions}.  For this reason we will study  a different
approach, which can be seen as a counterpart for complexes of the
method of de Smit and Lenstra in \cite{dSL}.  They first considered lifts of matrix representations of groups;
these are called framed deformations by Kisin in \cite[\S 2.3.4]{Kisin}.  One then considers
the natural morphism of functors from framed deformations to deformations.  When this idea is applied to 
complexes $V^\bullet$, a new issue arises:

\begin{question}
\label{q:fundfinite}
Is $U(G,V^\bullet)$ represented by a complex of modules for the completed
group ring $R(G,V^\bullet)[[G]]$ that is strictly perfect as a complex of $R(G,V^\bullet)$-modules? 
\end{question}

Here a strictly perfect complex of $R(G,V^\bullet)$-modules is a bounded complex of
finitely generated projective $R(G,V^\bullet)$-modules.
The answer to Question \ref{q:fundfinite} is yes (and obvious) when $V^\bullet$ has only 
one non-zero cohomology group, corresponding to the
classical case.  But we do not know the answer in general, even when
$V^\bullet$ has only two non-zero cohomology groups.   

 We view
Question \ref{q:fundfinite}  as a finiteness problem because when $G$ is topologically finitely generated,
a complex of $R(G,V^\bullet)[[G]]$-modules
representing $U(G,V^\bullet)$ that  is strictly perfect as a complex of $R(G,V^\bullet)$-modules 
can be described by a finite number of finite matrices
with coefficients in $R(G,V^\bullet)$.   An a priori result showing the existence of such
a description, especially with explicit bounds on the sizes of the matrices,
 can be very useful in determining the ring $R(G,V^\bullet)$ via matrices with indeterminate entries.   The
 proof of 
 Theorem \ref{thm:GabBigOne} gives an example of this method.  

It is not very difficult to show that under our hypotheses on $G$, there is a theory
of framed deformations for $V^\bullet$, in the following sense.  One can represent
$V^\bullet$ by a fixed choice of a bounded complex of pseudocompact $k[[G]]$-modules each of
which is finite dimensional over $k$.  Fix a choice of ordered $k$-basis for
each term of $V^\bullet$.  By a framed deformation over $A$ one means a 
complex of pseudocompact $A[[G]]$-modules $M^\bullet$ along with ordered bases for 
the terms of $M^\bullet$ as free finitely generated $A$-modules such that there is an 
isomorphism of complexes $k \hat{\otimes}_A M^\bullet \to V^\bullet$
which carries the chosen ordered bases for the terms of $M^\bullet$ to the chosen ordered bases
for the terms of $V^\bullet$.  Isomorphisms of framed deformations must be isomorphisms
of complexes which respect ordered bases.  One can show, using Schlessinger's criteria,
that under the hypotheses on $G$ we make in \S\ref{s:lifts}, there is a versal deformation
ring for the resulting functor. There is a natural transformation from this
framed deformation functor to the functor $\hat{F}_{V^\bullet}$. The issue
in Question \ref{q:fundfinite} is whether this natural transformation will be surjective
if we choose the ranks of the terms of $V^\bullet$ to be sufficiently large.  This
amounts to asking whether a single framed deformation functor has the
derived category deformation functor $\hat{F}_{V^\bullet}$ as a quotient.

It is not hard to show that $U(G,V^\bullet)$
is represented by a bounded above complex of projective modules for  $R(G,V^\bullet)[[G]]$. 
The difficulty is that the standard results concerning truncations
of such complexes do not readily produce quasi-isomorphic complexes of 
$R(G,V^\bullet)[[G]]$-modules that
are strictly perfect as complexes of modules for $R(G,V^\bullet)$, which is a much smaller ring than 
$R(G,V^\bullet)[[G]]$.

A fundamental problem 
in the subject appears to us to be whether Question \ref{q:fundfinite}
always has an affirmative answer if $V^\bullet$ arises from arithmetic,
in a suitable sense.  We will prove the following result  concerning this question:

\begin{thm}
\label{thm:yesitis}
Suppose $G$ is either 
\begin{enumerate}
\item[(i)] topologically finitely generated and abelian, or 
\item[(ii)] the tame fundamental group
of the spectrum of a regular local ring $S$ whose residue field is finite
of characteristic different from $p$ with respect to a divisor with strict
normal crossings. 
\end{enumerate}  
Then $U(G,V^\bullet)$ is represented
by a complex of $R(G,V^\bullet)[[G]]$-modules that is strictly perfect as a complex of 
$R(G,V^\bullet)$-modules.
\end{thm}  

In  \S \ref{s:ex} we will apply this Theorem to compute $U(G,V^\bullet)$ and $R(G,V^\bullet)$
for some natural examples in which $S$ in Theorem \ref{thm:yesitis}
is the $\ell$-adic integers $\mathbb{Z}_\ell$ for some prime $\ell\neq p$.  These
examples pertain to the deformation of elements of order $2$ in the Brauer
group of $\mathbb{Q}_\ell$.  Examples of this kind were first
considered in \cite{bcderived}, where we determined the associated universal 
flat deformation rings.
We will produce some examples in which the versal deformation ring 
is strictly larger than the versal proflat deformation ring.  Finding
explicit arithmetic constructions of the associated  versal deformations leads
to interesting number theoretical questions, and is a good test
of any general theory for determining deformations of complexes of
modules for a profinite group.  

We now give an outline of this paper. 

In \S \ref{s:lifts} we recall the definitions needed to state the main result of  \cite{bcderived} 
concerning the existence of versal and universal deformations of objects $V^\bullet $ in $D^-(k[[G]])$.  
In \S \ref{s:finiteness} we give a proof of Theorem \ref{thm:yesitis}.
The argument outlined in \S \ref{ss:outlinemain} proceeds by
improving the representative for the versal deformation in question
by three steps.  In the first step one works from right to left to
produce a complex whose individual terms have large annihilators.
In the second step, one works from left to right and uses an
Artin-Rees argument to produce a complex whose terms are
finitely generated over the versal deformation ring.  Finally in
the last step one works from right to left to refine these terms
so they become finitely generated and projective over the versal
deformation ring.  In  \S \ref{s:ex} we conclude with some examples
 pertaining to the element of order $2$ in the Brauer
group of $\mathbb{Q}_\ell$.

\medbreak
\noindent {\bf Acknowledgments:}
The authors would like to thank Luc Illusie and Ofer Gabber for valuable discussions, and 
the Banff International Research Station for support during the preparation of part of this paper. 
\medbreak


\section{Quasi-lifts and deformation functors}
\label{s:lifts}
\setcounter{equation}{0}

Let $G$ be a profinite group, let $k$ be a field of characteristic $p>0$,
and let $W$ be a complete local commutative Noetherian ring with residue field $k$.
Define
$\hat{\mathcal{C}}$ to be the category of complete local commutative Noetherian 
$W$-algebras
with residue field $k$. The morphisms in $\hat{\mathcal{C}}$ are 
continuous $W$-algebra
homomorphisms that induce the identity on $k$.
Let $\mathcal{C}$ be the subcategory of  Artinian objects in $\hat{\mathcal{C}}$.

Let $R \in \mathrm{Ob}(\hat{\mathcal{C}})$. Then $R[[G]]$ denotes the completed group algebra of the 
usual abstract group algebra $R[G]$ of $G$ over  $R$, i.e. $R[[G]]$ is the projective limit of the ordinary group algebras $R[G/U]$ as $U$ ranges over the open normal subgroups of $G$. 
We have that  $R$ is a pseudocompact ring and $R[[G]]$ is a pseudocompact $R$-algebra. 

Pseudocompact rings, algebras and modules have been studied, for example, in 
\cite{ga1,ga2,brumer}. Recall that a pseudocompact ring is a topological ring $\Lambda$ that is
complete and Hausdorff and admits a basis of open neighborhoods of $0$ consisting 
of two-sided ideals $J$ for which $\Lambda/J$ is an Artinian ring. Let $\Lambda$ be a pseudocompact
ring. Then $\Lambda$ is the projective limit of Artinian quotient rings having the discrete topology. 
A pseudocompact $\Lambda$-module is a complete Hausdorff topological $\Lambda$-module
$M$ which has a basis of open neighborhoods  of $0$ consisting of submodules $N$ for which
$M/N$ has finite length as $\Lambda$-module. Put differently, a $\Lambda$-module is 
pseudocompact if and only if it is the projective limit of $\Lambda$-modules of finite length having 
the discrete topology. 
If $R$ is a commutative pseudocompact ring and $\Lambda$ is a complete Hausdorff topological ring, 
then $\Lambda$ is called a pseudocompact $R$-algebra provided $\Lambda$ is an $R$-algebra in
the usual sense and $\Lambda$ admits a basis of open neighborhoods of $0$ consisting of 
two-sided ideals $J$ for which $\Lambda/J$ has finite length as $R$-module. 
Note that every pseudocompact $R$-algebra is a pseudocompact ring, and a module over 
a pseudocompact $R$-algebra has finite length if and only if it has finite length as $R$-module.

\begin{rem}
\label{rem:newandimproved}
Let $\Lambda$ be a pseudocompact ring and let $R$ be a commutative pseudocompact ring.
Denote the category of pseudocompact left $\Lambda$-modules by $\mathrm{PCMod}(\Lambda)$.

Recall that a pseudocompact $\Lambda$-module $M$ is said to be topologically free on a set
$X=\{x_i\}_{i\in I}$ if $M$ is isomorphic to the product of a family $(\Lambda_i)_{i\in I}$ where
$\Lambda_i=\Lambda$ for all $i$.
\begin{enumerate}
\item[(i)] The category $\mathrm{PCMod}(\Lambda)$ is an abelian category with exact projective limits.
Since every topologically free pseudocompact $\Lambda$-module is a projective object in 
$\mathrm{PCMod}(\Lambda)$ and since every pseudocompact
$\Lambda$-module is the quotient of a topologically free $\Lambda$-module,
$\mathrm{PCMod}(\Lambda)$ has enough projective objects. 

\item[(iii)] 
If $M$ and $N$ are pseudocompact $\Lambda$-modules, then we define the right derived functors
$\mathrm{Ext}^n_{\Lambda}(M,N)$ by using a projective resolution of $M$. 

\item[(iv)] 
Suppose $\Lambda$ is a pseudocompact $R$-algebra, and let $\hat{\otimes}_{\Lambda}$ denote 
the completed tensor product in the category $\mathrm{PCMod}(\Lambda)$ (see \cite[\S 2]{brumer}).
If $M$ is a right (resp. left) pseudocompact $\Lambda$-module, then $M\hat{\otimes}_\Lambda- $ 
(resp. $-\hat{\otimes}_\Lambda M$) is a right exact functor.

If $M$ is finitely generated as a pseudocompact $\Lambda$-module, it follows from
\cite[Lemma 2.1(ii)]{brumer} that the functors
$M \otimes_\Lambda -$ and $M\hat{\otimes}_\Lambda -$ (resp. $-\otimes_\Lambda M$ and 
$-\hat{\otimes}_\Lambda M$) are naturally isomorphic.

\item[(v)] 
Suppose $\Lambda$ is a pseudocompact $R$-algebra and $M$ is a right (resp. left) pseudocompact 
$\Lambda$-module. Recall that $M$ is said to be  topologically flat, if the functor 
$M\hat{\otimes}_{\Lambda}-$ (resp. $-\hat{\otimes}_\Lambda M$) is exact.
By \cite[Lemma 2.1(iii)]{brumer} and \cite[Prop. 3.1]{brumer}, $M$ is topologically flat if and only 
if $M$ is projective.

If $\Lambda=R$ and $M$ is a pseudocompact $R$-module,
it follows from \cite[Proof of Prop. 0.3.7]{ga2} and \cite[Cor. 0.3.8]{ga2} that $M$ is 
topologically flat if and only if $M$ is topologically free if and only if $M$ is abstractly flat.
In particular, if $R$ is Artinian, a pseudocompact $R$-module is topologically flat  if and only
if it is abstractly free.

\end{enumerate}
\end{rem}

\begin{rem}
\label{rem:useful}
Let $\Lambda$ be a pseudocompact ring.
\begin{enumerate}
\item[(i)] Suppose $f:M\to N$ is a homomorphism
of pseudocompact $\Lambda$-modules. Since $\mathrm{PCMod}(\Lambda)$ has exact
projective limits, it follows that the image of $f$ is closed in $N$ and is therefore 
a pseudocompact $\Lambda$-submodule of $N$. 

In particular, if $I$ is a two-sided ideal of $\Lambda$ and $N$ is a pseudocompact 
left $\Lambda$-module
such that both $I$ and $N$ are finitely generated as abstract left $\Lambda$-modules, then
$I\, N$ is a closed pseudocompact $\Lambda$-submodule of $N$, since it is the image of a
homomorphism $f:M\to N$ of pseudocompact $\Lambda$-modules in which $M$ is a topologically
free pseudocompact $\Lambda$-module on a finite set of cardinality equal to the product of
the cardinalities of generating sets for $I$ and $N$.

\item[(ii)] By \cite[Lemma 1.1]{brumer}, if $f:M\to N$ is an epimorphism in
$\mathrm{PCMod}(\Lambda)$, i.e. a surjective homomorphism of pseudocompact
$\Lambda$-modules, then there is a continuous section $s:N\to M$
such that $f\circ s$ is the identity morphism on $N$. In particular, a
homomorphism $f:M\to N$ of pseudocompact $\Lambda$-modules is an isomorphism in 
$\mathrm{PCMod}(\Lambda)$ if and only if it is bijective.
\item[(iii)] Suppose $M$ is a pseudocompact $\Lambda$-module that
is free and finitely generated as an abstract $\Lambda$-module. Since
a topologically free pseudocompact $\Lambda$-module on a finite set $X$ is
isomorphic to an abstractly free $\Lambda$-module on $X$, one sees that
$M$ is a topologically free pseudocompact $\Lambda$-module on a finite set.
\end{enumerate}
\end{rem}

\medskip

If $\Lambda$ is a pseudocompact ring, let $C^-(\Lambda)$
be the abelian category of complexes of pseudocompact $\Lambda$-modules that are bounded above, let $K^-(\Lambda)$ be the homotopy category of $C^-(\Lambda)$, and
let $D^-(\Lambda)$ be the derived category of $K^-(\Lambda)$. 
Let $[1]$ denote the translation functor on $C^-(\Lambda)$ (resp. $K^-(\Lambda)$, 
resp. $D^-(\Lambda)$), 
i.e. $[1]$ shifts complexes
one place to the left and changes the sign of the differential.
Note that by Remark \ref{rem:useful}(ii), a homomorphism in $C^-(\Lambda)$
is a quasi-isomorphism if and only if the induced homomorphisms on 
all the cohomology groups are bijective.

\begin{hypo}
\label{hypo:fincoh}
Throughout this paper, we assume that $V^\bullet $ is a
complex in $D^-(k[[G]])$ that has  only finitely many non-zero cohomology
groups, all of which have finite $k$-dimension.
\end{hypo}

\begin{rem}
\label{rem:leftderivedtensor}
Let  $X^\bullet, Y^\bullet\in \mathrm{Ob}(K^-(R[[G]]))$ and consider the double complex $K^{\bullet,\bullet}$ of pseudocompact $R[[G]]$-modules with
$K^{p,q}=(X^p\hat{\otimes}_RY^q)$ and diagonal $G$-action. We define the total tensor product $X^\bullet\hat{\otimes}_R Y^\bullet$ to be the simple complex associated to $K^{\bullet,\bullet}$, i.e.
$$(X^\bullet\hat{\otimes}_RY^\bullet)^n=\bigoplus_{p+q=n}X^p\hat{\otimes}_RY^q$$
whose differential is $d(x\,\hat{\otimes}\, y)=d_X(x)\,\hat{\otimes} \,y + (-1)^x \,x\,\hat{\otimes} \,d_Y(y)$
for $x\,\hat{\otimes}\,y\in K^{p,q}$.
Since homotopies carry over the completed tensor product, we have a functor
$$\hat{\otimes}_R : K^-(R[[G]])\times K^-(R[[G]])\to K^-(R[[G]]).$$
Using \cite[Thm. 2.2 of Chap. 2 \S2]{verdier}, we see that there is a well-defined
left derived completed tensor product $\hat{\otimes}^{\LL}_R$. Moreover, if
$X^\bullet$ and $Y^\bullet$ are as above, then $X^\bullet\hat{\otimes}^{\LL}_R Y^\bullet$
may be computed in  $D^-(R[[G]])$ in the following way.  
Take a bounded above complex ${Y'}^\bullet$ of topologically flat pseudocompact $R[[G]]$-modules
with a quasi-isomorphism ${Y'}^\bullet\to Y^\bullet$ in $K^-(R[[G]])$. Then this quasi-isomorphism 
induces an isomorphism between
$X^\bullet\hat{\otimes}_R{Y'}^\bullet$ and $X^\bullet\hat{\otimes}^{\LL}_RY^\bullet$ in $D^-(R[[G]])$.
\end{rem}

\begin{dfn}
\label{def:quasilifts}
\begin{enumerate}
\item[(a)]
We will say that a complex $M^\bullet$ in $K^-(R[[G]])$
has \emph{finite pseudocompact $R$-tor dimension}, 
if there exists an integer $N$ such that for all pseudocompact
$R$-modules $S$, and for all integers $i<N$, ${\HH}^i(S\hat{\otimes}^{\LL}_R M^\bullet)=0$.
If we want to emphasize the integer $N$ in this definition, we say $M^\bullet$ has 
\emph{finite pseudocompact $R$-tor dimension at $N$}.

\item[(b)]
A \emph{quasi-lift}
of $V^\bullet$ over an object $R$ of
$\hat{\mathcal{C}}$ is a pair $(M^\bullet,\phi)$ consisting of a complex
$M^\bullet$ in
$D^-(R[[G]])$ that has finite pseudocompact $R$-tor  dimension
together with an isomorphism
\hbox{$\phi: k \hat{\otimes}^{\LL}_R M^\bullet \to V^\bullet$} in $D^-(k[[G]])$.
Two
quasi-lifts $(M^\bullet, \phi)$ and $({M'}^\bullet,\phi')$ are
\emph{isomorphic} if there is an isomorphism
$f:M^\bullet \to {M'}^\bullet$ in $D^-(R[[G]])$ with
$\phi'\circ(k\hat{\otimes}^{\LL}f)=\phi$.

\item[(c)]
Let $\hat{F} = \hat{F}_{V^\bullet}:\hat{\mathcal{C}} \to \mathrm{Sets}$
be the functor which sends an object $R$ of $\hat{\mathcal{C}}$ to the set
$\hat{F}(R)$ of all isomorphism classes of quasi-lifts 
of $V^\bullet$ over $R$, and which sends
a morphism $\alpha:R\to R'$ in $\hat{\mathcal{C}}$ to the set map
$\hat{F}(R)\to \hat{F}(R')$ induced by $M^\bullet \mapsto R'\hat{\otimes}_{R,\alpha}^{\LL}
M^\bullet$. Let $F = F_{V^\bullet}$ be the restriction of $\hat{F}$ 
to the subcategory $\mathcal{C}$ of Artinian objects in $\hat{\mathcal{C}}$.

Let $k[\varepsilon]$, where $\varepsilon^2=0$, denote the ring of dual numbers over
$k$. The set $F(k[\varepsilon])$ is called the \emph{tangent space} to 
$F$, denoted by $t_{F}$. 
\end{enumerate}
\end{dfn}

\begin{rem}
\label{rem:dumbdumb}
Suppose $M^\bullet$ is a complex in $K^-([[RG]])$ of topologically flat, hence topologically free, 
pseudocompact $R$-modules that has finite pseudocompact $R$-tor dimension
at $N$. Then the bounded complex ${M'}^\bullet$, which is
obtained from $M^\bullet$ by replacing $M^N$ by
${M'}^N=M^N/\delta^{N-1}(M^{N-1})$ and by setting ${M'}^i = 0$ if $i < N$,
is quasi-isomorphic to $M^\bullet$ and
has topologically free pseudocompact terms over $R$.
\end{rem}

\begin{thm} 
\label{thm:derivedresult}  
Suppose that $\HH^i(V^\bullet) = 0$ unless $n_1 \le i \le n_2$.  Every quasi-lift of $V^\bullet$ over 
an object $R$ of $\hat{\mathcal{C}}$ is isomorphic to a quasi-lift $(P^\bullet, \phi)$ for
a complex $P^\bullet$ with the following properties:
\begin{enumerate}
\item[(i)] The terms of $P^\bullet$ are topologically free $R[[G]]$-modules.
\item[(ii)] The cohomology group $\HH^i(P^\bullet)$ is finitely generated 
as an abstract $R$-module for all $i$, and $\HH^i(P^\bullet) = 0$ unless $n_1 \le i \le n_2$. 
\item[(iii)]   One has $\HH^i(S\hat{\otimes}^{\LL}_RP^\bullet)=0$  for all pseudocompact $R$-modules 
$S$ unless $n_1 \le i \le n_2$.
\end{enumerate}
\end{thm}

\begin{proof}
Part (i)  follows from \cite[Lemma 2.9]{bcderived}. Assume now
that the terms of $P^\bullet$ are topologically free $R[[G]]$-modules, which means in particular
that the functors $-\hat{\otimes}^{\LL}_RP^\bullet$ and $-\hat{\otimes}_RP^\bullet$ are
naturally isomorphic. Let $m_R$ denote the maximal ideal of $R$, and let $n$ be an
arbitrary positive integer.
By \cite[Lemmas 3.1 and 3.8]{bcderived}, $\HH^i((R/m_R^n)\hat{\otimes}_RP^\bullet)=0$ for $i>n_2$ 
and $i<n_1$. Moreover, for $n_1\le i\le n_2$,
$\HH^i((R/m_R^n)\hat{\otimes}_RP^\bullet)$ is a subquotient of  an abstractly free $(R/m_R^n)$-module
of rank $d_i=\mathrm{dim}_k\,\HH^i(V^\bullet)$,
and $(R/m_R^n)\hat{\otimes}_RP^\bullet$ has finite pseudocompact $(R/m_R^n)$-tor dimension at
$N=n_1$. Since 
$P^\bullet\cong \displaystyle \lim_{\stackrel{\longleftarrow}{n}}\, (R/m_R^n)\hat{\otimes}_RP^\bullet$ 
and since by Remark \ref{rem:newandimproved}(i), the category $\mathrm{PCMod}(R)$ has 
exact projective limits, it  follows that for all pseudocompact $R$-modules $S$
$$\HH^i(S\hat{\otimes}_RP^\bullet)= \lim_{\stackrel{\longleftarrow}{n}}\,\HH^i\left(
(S/m_R^nS)\hat{\otimes}_{R/m_R^n}\left((R/m_R^n)\hat{\otimes}_R P^\bullet\right)\right)$$ 
for all $i$. Hence Theorem \ref{thm:derivedresult} follows.
\end{proof}

\begin{dfn}
\label{dfn:bndcoh}
A profinite group $G$ has \emph{finite pseudocompact cohomology},
if for each discrete $k[[G]]$-module $M$ of finite $k$-dimension,
and all integers $j$, the cohomology group ${\HH}^j(G,M)=\mathrm{Ext}^j_{k[[G]]}(k,M)$ 
has finite $k$-dimension.
\end{dfn}

\begin{thm}
\label{thm:bigthm} {\rm (\cite[Thm.  2.14]{bcderived})}
Suppose that $G$ has finite pseudocompact cohomology.
\begin{enumerate}
\item[(i)] 
The functor
$F$ has a pro-representable hull 
$R(G,V^\bullet)\in \mathrm{Ob}(\hat{\mathcal{C}})$ 
$($c.f. \cite[Def. 2.7]{Sch} and \cite[\S 1.2]{Maz}$)$, and 
the functor
$\hat{F}$ is continuous
$($c.f. \cite{Maz}$)$. 

\item[(ii)]
There is a $k$-vector space isomorphism $h: t_F \to
\mathrm{Ext}^1_{D^-(k[[G]])}(V^\bullet,V^\bullet)$.

\item[(iii)]
If $\mathrm{Hom}_{D^-(k[[G]])}(V^\bullet,V^\bullet)= k$, then $\hat{F}$ is represented
by $R(G,V^\bullet)$. 
\end{enumerate}
\end{thm}

\begin{rem}
\label{rem:newrem}
By Theorem \ref{thm:bigthm}(i), there exists 
a quasi-lift $(U(G,V^\bullet),\phi_U)$ of $V^\bullet$ over $R(G,V^\bullet)$ 
with the following property. For each $R\in \mathrm{Ob}(\hat{\mathcal{C}})$, the map
$\mathrm{Hom}_{\hat{\mathcal{C}}}(R(G,V^\bullet),R) \to \hat{F}(R)$ 
induced by $\alpha \mapsto R\hat{\otimes}^{\LL}_{R(G,V^\bullet),\alpha} U(G,V^\bullet)$ is surjective,
and this map is bijective if $R$ is the ring of dual numbers $k[\varepsilon]$ over $k$
where $\varepsilon^2=0$.

In general,
the isomorphism type of the pro-representable hull $R(G,V^\bullet)$ is
unique up to non-canonical isomorphism.
If $R(G,V^\bullet)$ represents $\hat{F}$,
then $R(G,V^\bullet)$
is uniquely determined up to canonical isomorphism.
\end{rem}

\begin{dfn}
\label{def:newdef}
Using the notation of Theorem \ref{thm:bigthm} and Remark \ref{rem:newrem}, 
we call 
$R(G,V^\bullet)$
the \emph{versal deformation ring}
of $V^\bullet$ and  $(U(G,V^\bullet),\phi_U)$ 
a \emph{versal deformation} of $V^\bullet$.

If $R(G,V^\bullet)$ represents $\hat{F}$, then
$R(G,V^\bullet)$ will be
called the \emph{universal deformation ring} 
of $V^\bullet$ and $(U(G,V^\bullet),\phi_U)$  will be called a
\emph{universal deformation} of $V^\bullet$.
\end{dfn}

\begin{rem}
\label{rem:bigthm}
If $V^\bullet$ consists of a single module $V_0$ in dimension $0$,
the versal deformation ring $R(G,V^\bullet)$ coincides with the versal deformation
ring studied by Mazur in \cite{maz1,Maz}.
In this case,
Mazur assumed only that $G$ satisfies a certain finiteness condition
($\Phi_p$), which is equivalent
to the requirement that ${\HH}^1(G,M)$ have finite $k$-dimension for all discrete 
$k[[G]]$-modules $M$ of finite $k$-dimension.
Since the higher $G$-cohomology enters into determining lifts of
complexes $V^\bullet$
having more than one non-zero cohomology group, the condition that $G$
have finite pseudocompact
cohomology is the natural generalization of Mazur's finiteness condition
in this context.
\end{rem}

We conclude this section by recalling a result from \cite{bcbigobstructions}.

\begin{prop}
\label{prop:prop}
{\rm (\cite[Prop. 4.2]{bcbigobstructions})}
Suppose $G$ has finite pseudocompact cohomology and $K$ is a closed normal subgroup of 
$G$ which is a pro-$p'$ group, i.e. the projective limit of finite groups that have order prime to $p$. 
Let $\Delta=G/K$, and suppose $V^\bullet$ is isomorphic to the inflation
$\mathrm{Inf}_\Delta^G\,V_\Delta^\bullet$ of a bounded above complex
$V_\Delta^\bullet$ of pseudocompact $k[[\Delta]]$-modules. Then 
 the two deformation functors 
$\hat{F}^G=\hat{F}^G_{V^\bullet}$ and $\hat{F}^\Delta=\hat{F}^\Delta_{V_\Delta^\bullet}$ which are defined according
to Definition $\ref{def:quasilifts}(c)$ are naturally isomorphic. In consequence,
 $R(G,V^\bullet)\cong R(\Delta,V_\Delta^\bullet)$
and $(U(G,V^\bullet),\phi_U)\cong
(\mathrm{Inf}_{\Delta}^G\,U(\Delta,V_\Delta^\bullet),\mathrm{Inf}_{\Delta}^G\,\phi_U)$.
\end{prop}


\section{Finiteness questions}
\label{s:finiteness}
\setcounter{equation}{0}

In this section, we consider the question of when every quasi-lift of $V^\bullet$ over a ring $A$ in
$\hat{\mathcal{C}}$ can be represented by a bounded complex of 
abstractly finitely generated free 
$A$-modules with continuous actions by $G$. Recall from Remark \ref{rem:useful}(iii)
that if a pseudocompact module is abstractly finitely generated free,
then it is topologically free on a finite set.
As before, $k$ has positive characteristic $p$. 
We distinguish two cases:

\medbreak

\noindent {\bf  Case A:}  $G$ is topologically finitely generated and abelian;  and 

\medbreak

\noindent {\bf Case B:}  $G$ is the tame fundamental group of the spectrum
of a regular local ring $S$ whose residue field $k(S)$ is finite
of characteristic $\ell \ne p$
with respect to a divisor $D$ with strict normal crossings.

\medbreak

We recall the structure of $G$ as in case B (see \cite{grothmur,Schmidt}).  Let
$Y = \mathrm{Spec}(S)$, and let $D_Y = D=\sum_{i = 1}^r \mathrm{div}_Y(f_i)$ for a subset
$\{f_i\}_{i = 1}^r$ of a system of local parameters for the maximal ideal $m_S$ of $S$.
Let $X = \mathrm{Spec}(S^h)$ be the strict henselization of $Y$, so that $S^h$ is
local, its residue field is equal to  the separable closure $k(S^h) = k(S)^s$ of $k(S)$, and $m_{S^h}$
is generated by $m_S$.   The divisor $D_X = \sum_{i = 1}^r \mathrm{div}_X(f_i)$ has
normal crossings on $X$.  We have an exact sequence
\begin{equation}
\label{eq:exactlydumb}
1 \to \pi_1^t(X,D_X) \to \pi_1^t(Y,D_Y) \to \mathrm{Gal}(k(S)^s/k(S)) \to 1
\end{equation}
in which $G=\pi_1^t(Y,D_Y)$ and $\pi_1^t(X,D_X)$ are tame fundamental groups.
There is a Kummer isomorphism
\begin{equation}
\label{eq:Kummeriso}
\pi_1^t(X,D_X) \cong \prod_{i = 1}^r \hat{\mathbb{Z}}^{(\ell ')}(1)
\end{equation}
in which $ \hat{\mathbb{Z}}^{(\ell ')}(1) = \displaystyle\lim_{\stackrel{\longleftarrow}{\ell \not \; | m}} \mu_m$.
The group $\mathrm{Gal}(k(S)^s/k(S))$ is procyclic and is topologically
generated by the Frobenius automorphism
$ \Phi_{k(S)}$ relative to the finite field $k(S)$.  Explicitly, if we define $f \ge 1$ so $k(S) = \mathbb{F}_{\ell^f}$
has order $\ell^f$, then a lift $\Phi \in \pi_1^t(Y,D_Y)$ of $\Phi_{k(S)}$ acts on each 
factor of $\pi_1^t(X,D_X)$ which is isomorphic to $\hat{\mathbb{Z}}^{(\ell ')}(1)$
via the map $\zeta \to \zeta^{\ell^f}$.  Since the procyclic group 
$\langle \Phi \rangle$, which is topologically generated by the lift $\Phi$, is isomorphic to the profinite completion $\hat {\mathbb{Z}}$ of $\mathbb{Z}$,
and this maps isomorphically to $\mathrm{Gal}(k(S)^s/k(S))$, we see that (\ref{eq:exactlydumb})
is a split exact sequence and $G=\pi_1^t(Y,D_Y)$ is the semidirect product of $\langle \Phi \rangle$
with $\pi_1^t(X,D_X)$.

Suppose $V^\bullet \in D^-(k[[G]])$ is as in Hypothesis \ref {hypo:fincoh}, i.e. $V^\bullet$
has  only finitely many non-zero cohomology groups, all of which have finite $k$-dimension.
Assume that $\HH^i(V^\bullet) = 0$ unless $n_1 \le i \le n_2$.  
Theorem \ref{thm:yesitis} states that for $G$ as in case A or case B, the versal deformation
$(U(G,V^\bullet),\phi_U)$ is represented in $D^-(R(G,V^\bullet)[[G]])$ 
by a complex that is strictly perfect as a complex of $R(G,V^\bullet)$-modules.
This is a consequence of the following result.

\begin{thm}
\label{thm:mainthm} 
Let $A $ be an object of $ \hat{\mathcal{C}}$.
Suppose $(P^\bullet,\phi)$ is a quasi-lift of $V^\bullet$ over $A$ such that $P^\bullet$ has properties $(i)$, $(ii)$ 
and $(iii)$ of Theorem $\ref{thm:derivedresult}$.  There is a bounded complex 
$Q^\bullet $ of pseudocompact $A[[G]]$-modules which is isomorphic to 
$P^\bullet$ in $D^-(A[[G]])$  for which each term $Q^i$ 
is an abstractly finitely generated free $A$-module 
and $Q^i = 0 $ unless $n_1 \le i \le n_2$.
\end{thm}

The proof of Theorem \ref{thm:mainthm}  is outlined in the next section and
carried out in subsequent sections.  


\subsection{Outline of the proof of Theorem \ref{thm:mainthm}}
\label{ss:outlinemain}

We begin with a reduction.

\begin{lemma}
\label{lem:efissocute} 
There is a pro-$p'$ closed normal subgroup $K$ of $G$ with the following properties:
\medbreak
\begin{enumerate}
\item[(i)]  The complex $V^\bullet$ is inflated from a complex for $\Delta = G/K$.
\medbreak
\item[(ii)] In case A, $\Delta = \mathbb{Z}_p^s \times Q \times Q'$ where $Q$ $($resp. $Q'$$)$ is a finite abelian $p$-group
$($resp. $p'$-group$)$.   
Let $w_{1,j}$ for $1 \le j \le s$ be topological generators for the $\mathbb{Z}_p$-factors in this description.
\medbreak
\item[(iii)] In case B, let $\hat{\mathbb{Z}}^{(\ell',p')}(1)$ be the unique maximal pro-$p'$ subgroup of  $\hat{\mathbb{Z}}^{(\ell ')}(1)$. Let $K_1$ be the maximal subgroup of 
$$N_1 = \prod_{i = 1}^r \hat{\mathbb{Z}}^{(\ell ',p')}(1)\subset  \prod_{i = 1}^r \hat{\mathbb{Z}}^{(\ell ')}(1) = \pi_1^t(X,D_X)$$
that acts trivially on all of the terms of $V^\bullet$.
Then $K_1$ is closed and normal in $G$ and 
$$\Delta_1 = \pi_1^t(X,D_X)/K_1 = \left ( \prod_{i = 1}^r \mathbb{Z}_p(1)\right ) \times \tilde{\Delta}_1$$
for a finite abelian group $\tilde{\Delta}_1$ which is of order prime to $p$ and $\ell$.  
Let $N_0 \subset \langle \Phi \rangle$ be the kernel of the action
of $\langle \Phi \rangle$ on $\Delta_1$, and view $\langle \Phi \rangle$ as a subgroup of $G$ via a 
choice of Frobenius $\Phi$ in $G$.  Define  $K_0$ to be the maximal subgroup of 
$N_0$ that acts trivially on all of the terms of $V^\bullet$.
The group $K$ generated by $K_0$ and $K_1$ is the semidirect product $K_1.K_0$
and is normal in $G$.
The group $\Delta = G/K$ is the semidirect product of 
$\Delta_1$ with the quotient $\Delta_0=\langle \Phi \rangle/K_0$.  
Let $\overline{\Phi}$ be the image of $\Phi$ in $\Delta_0$.  
The group $\Delta_0$ is isomorphic to the
product $\langle \overline{\Phi}^{d} \rangle \times \tilde{\Delta}_0$, where  $\tilde{\Delta}_0$ is cyclic of order $d$ 
prime to $p$ and $\langle \overline{\Phi}^{d}\rangle $ is isomorphic to $\mathbb{Z}_p$. 
Define $w_1 = \overline{\Phi}^{d}$, and let $\{w_{2,j}\}_{j = 1}^r$ be topological generators for the 
$\mathbb{Z}_p(1)$-factors in  $\Delta_1$.  
\end{enumerate}
\end{lemma}

\begin{proof}  Case A is clear, since $G$ is abelian in this case.  
To prove this for $G= \pi_1^t(Y,D_Y)$ as in case B, one 
lets $d'$ be the smallest integer such that $\ell^{fd'} \equiv 1$ mod $p$.
In particular, $d'$ is relatively prime to $p$.
Writing $\langle \Phi^{d'} \rangle$ as a product $d'\,\prod_{q} \mathbb{Z}_q$ 
as $q$ ranges over all primes, one 
shows that the kernel of the action
of $\langle \Phi \rangle$ on $\mathbb{Z}_p(1)$ is 
equal to $d'\,\prod_{q \ne p} \mathbb{Z}_q$. It follows that
$N_0$ is  the subgroup of $d' \prod_{q \ne p} \mathbb{Z}_q$ that acts trivially on the
characteristic subgroup $\tilde{\Delta}_1$ of $\Delta_1$.  
Since $\tilde{\Delta}_1$ is finite and $K_0$ is the maximal subgroup of 
$N_0$ that acts trivially on all of the terms of $V^\bullet$, this implies that $K_0$ 
has finite index in $d' \prod_{q \ne p} \mathbb{Z}_q$.
Thus $K_0$ has finite index $d$ which is prime to $p$ in $\prod_{q \ne p} \mathbb{Z}_q$. 
One obtains that
$$\langle \Phi \rangle = \mathbb{Z}_p \times \left (\prod_{q \ne p} \mathbb{Z}_q \right ) \supset \{0\} \times d \left (\prod_{q \ne p} \mathbb{Z}_q \right ) = K_0,$$
which proves that
$\Delta_0 = \langle \Phi \rangle / K_0 = \mathbb{Z}_p \times \tilde{\Delta}_0$,
where $\tilde{\Delta}_0$ is finite and cyclic of order $d$ prime to $p$.  
The remaining statements in the lemma now follow.
\end{proof}

The following result is a consequence of Lemma \ref{lem:efissocute} and Proposition \ref{prop:prop}.

\begin{cor}
\label{cor:reductioabsurdum}  
It suffices to prove Theorem $\ref{thm:mainthm}$ when $G$ is replaced
by the group $\Delta$ described in Lemma $\ref{lem:efissocute}$.
\end{cor}

Let $A$ be an object of $\hat{\mathcal{C}}$. 
Since $A[[\mathbb{Z}_p^s]]$ is isomorphic to a power
series algebra over $A$ in $s$ commuting variables, it follows that 
$A[[\Delta]]$ is left and right Noetherian 
for $\Delta$ as in Lemma \ref{lem:efissocute}(ii).
For $\Delta$ as in Lemma \ref{lem:efissocute}(iii),
one considers the subgroup $\tilde{\Delta}$ of finite index
in $\Delta$ that is topologically generated
by $w_1=\overline{\Phi}^d$ and by $w_{2,j}$, $1\le j\le r$. 
By embedding $\tilde{\Delta}$ as a closed subgroup of block diagonal matrices
with blocks of size $2$ inside $\mathrm{GL}_{2r}(\mathbb{Z}_p)$,
one sees that $\tilde{\Delta}$ is a compact $p$-adic analytic group.
Hence it follows from Lazard's result \cite[Prop. 2.2.4 of Chap. V]{Lazard}
that $\mathbb{Z}_p[[\tilde{\Delta}]]$ is left and right Noetherian.
Since Lazard's arguments also work if $\mathbb{Z}_p$ is replaced by $A$,
we obtain the following result.

\begin{lemma}
\label{lem:noetherian}
If $A$ is an arbitrary object of $\hat{\mathcal{C}}$ and $\Delta$ is as in Lemma 
$\ref{lem:efissocute}$,
then the ring $B=A[[\Delta]]$ is both left Noetherian and right Noetherian.
\end{lemma}

For the remainder of this section, let 
$A$ be an object of $\hat{\mathcal{C}}$ and let $B = A[[\Delta]]$.
To better explain the main ideas of the proof without having to use multiple
subscripts, we will at first assume that if $\Delta$ is as in Lemma \ref{lem:efissocute}(iii)
then $r=1$. In this case we will write $w_2$ instead of $w_{2,1}$.
We will show in 
\S \ref{ss:generalR} how to generalize the proofs to work for $r> 1$. 

The proof of Theorem \ref{thm:mainthm} depends on the following results.

\begin{prop}
\label{prop:zerostep}
Suppose $\Delta$ is as in Lemma $\ref{lem:efissocute}(iii)$ and $r=1$.
Define $w_2=w_{2,1}$.
For positive integers $N,N'$, let $J=B\cdot (w_2^N-1)^{N'}$. 
Then $J$ is a closed two-sided ideal of $B$ and
the quotient ring $\overline{B}=B/J$ is a pseudocompact $A$-algebra. 
Moreover, $J$ is a topologically free
rank one left $B$-module and a topologically free rank one 
right $B$-module.  
\end{prop}

\begin{prop}
\label{prop:firststep}
Suppose $\Delta$ is as in Lemma $\ref{lem:efissocute}(iii)$ and $r=1$.
Define $w_2=w_{2,1}$.
Let $M$ be a 
pseudocompact $B$-module that is finitely generated as an abstract $A$-module.
Then there exist positive
integers $N,N'$ such that $(w_2^N-1)^{N'}\cdot M=\{0\}$.  
\end{prop}

\begin{prop}
\label{prop:secondprop}  
Let $J$ be a two-sided ideal in $B$ of the following form:
\begin{enumerate}
\item[(i)] If $\Delta$ is as in Lemma $\ref{lem:efissocute}(ii)$, let $J=\{0\}$.
\item[(ii)] If $\Delta$ is as in Lemma $\ref{lem:efissocute}(iii)$ and $r=1$, 
let $J=B\cdot (w_2^N-1)^{N'}$, where $w_2=w_{2,1}$ and $N,N'$ are positive integers. 
\end{enumerate}
If $\Lambda=B/J$, then $\Lambda$ is a pseudocompact $A$-algebra.
Suppose $M$ is a pseudocompact $\Lambda$-module that is finitely
generated as an abstract $\Lambda$-module.
Let $T$ be a pseudocompact $\Lambda$-submodule of $M$
that is finitely generated as an abstract $A$-module. Then there is a 
pseudocompact
$\Lambda$-submodule $M'$ of $M$ such that $M' \cap T = \{0\}$ and $M/M'$ is 
finitely generated as an abstract $A$-module.
\end{prop}

\begin{prop}
\label{prop:bounded}
Let $\Omega$ be a pseudocompact ring that is left Noetherian.
Let $P^\bullet$ be a complex in $D^-(\Omega)$ whose terms $P^i$ are
free and finitely generated as abstract $\Omega$-modules 
such that $P^i = 0$ if $i > 0$.  Suppose that for $i \le 0$, $I_i$ is a closed two-sided 
ideal in $\Omega$ with the following properties. 
\begin{enumerate}
\item[(a)]  The cohomology group $\HH^i(P^\bullet)$ is annihilated by $I_i$ for $i \le 0$.
\item[(b)] For $i \le 0$, the  two-sided ideal $J_i=I_i \cdot I_{i+1} \cdots I_1 \cdot I_0$ is 
free and finitely generated as an abstract left $\Omega$-module.
\end{enumerate}
Then $P^\bullet$ is isomorphic in $D^-(\Omega)$ to a complex 
$Q^\bullet$ such that $Q^i = 0$ for $i > 0$ and $Q^i$ is annihilated by $J_i$ 
for $i \le 0$.  
\end{prop}

\begin{prop}
\label{prop:secondstep}
Suppose $\Delta$ is one of the groups in Lemma $\ref{lem:efissocute}$,
where we assume $r=1$ when $\Delta$ is as in Lemma $\ref{lem:efissocute}(iii)$.
Let $M$ be a pseudocompact $B$-module that is finitely generated
as an abstract $A$-module.
Then there exists a pseudocompact $B$-module $F$ that is 
free and finitely generated as an abstract $A$-module
and a surjective homomorphism $\varphi:F \to M$ of pseudocompact
$B$-modules.
\end{prop}

\begin{rem}
\label{rem:hartshorne}
Let $\Omega$ be a pseudocompact ring that is left Noetherian, and let
$M^\bullet$ be a bounded above complex of pseudocompact $\Omega$-modules
such that $M^i=0$ for $i>n$ 
and the cohomology groups $\HH^i(M^\bullet)$ are finitely generated
as abstract $\Omega$-modules. The construction given by Hartshorne in 
\cite[III Lemma 12.3]{Hartshorne} shows that 
there is a quasi-isomorphism $\rho:L^\bullet\to M^\bullet$ in $C^-(\Omega)$,
where $L^\bullet$ is a bounded above complex of pseudocompact 
$\Omega$-modules that are free and finitely generated as abstract 
$\Omega$-modules and $L^i=0$ for $i>n$. Moreover, we can require
$\rho^{n-1}:L^{n-1}\to M^{n-1}$ to be surjective.
\end{rem}

We first show how Theorem \ref{thm:mainthm} follows from these results when 
$G$ is replaced by $\Delta$ and, if $\Delta$ is as in Lemma \ref{lem:efissocute}(iii),
we assume $r=1$. As before, we write $w_2$ instead of $w_{2,1}$.

Suppose $P^\bullet$ has properties (i), (ii) and (iii) of Theorem \ref{thm:derivedresult}.
Without loss of generality we will suppose that $n_2 = 0$, so that $P^i = 0$ if $i > 0$.

\medbreak
 
\noindent \textit{Step $1$}: 
The complex $P^\bullet$ is isomorphic in $D^-(B)$ to a 
complex $Q^\bullet$ such that $Q^i = 0$ if $i > 0$ or $i<n_1$ and such that
if $n_1\le i\le 0$ then $Q^i$ is annihilated by a closed two-sided ideal $J$ in $B$
of the form described in Proposition \ref{prop:secondprop}.

\medbreak

\noindent\textit{Proof of Step $1$.}
If $\Delta$ is as in Lemma \ref{lem:efissocute}(ii), we can define
$Q^\bullet$ to be the complex obtained from $P^\bullet$ by replacing $P^{n_1}$ 
by $P^{n_1}/\BB^{n_1}(P^\bullet)$ and $P^i$ by $0$ for $i < n_1$.

Suppose now that $\Delta$ is as in Lemma \ref{lem:efissocute}(iii)
and $r=1$. Using Remark \ref{rem:hartshorne}, we can assume that
the terms of $P^\bullet$ are free and finitely generated as abstract
$B$-modules and that $P^i=0$ for $i>0$.
For $i\le 0$, we apply Proposition \ref{prop:firststep} to 
$M = \HH^i(P^\bullet)$ to 
see that there are integers $N(i), N'(i) \ge 1$ such that 
the left ideal $I_{i}=B\cdot(w_2^{N(i)} - 1)^{N'(i)}$ 
annihilates $\HH^i(P^\bullet)$.  Proposition \ref{prop:zerostep} shows that 
$I_{i}$ is a closed two-sided ideal of $B$ that is a
topologically free rank one right $B$-module and a topologically
free rank one left $B$-module. 
Therefore for $i \le 0$, 
the  ideal $J_i=I_i \cdot I_{i+1} \cdots I_1 \cdot I_0$
is a topologically and abstractly free rank one left $B$-module.  
The hypotheses of Proposition \ref{prop:bounded} are now satisfied 
when we let $\Omega = B$.
Therefore  $P^\bullet$ is isomorphic in $D^-(B)$ to a complex 
$Q^\bullet$ such that $Q^i = 0$ for $i > 0$ and $Q^i$ is annihilated by 
$J_i$ for $i \le 0$.  Since $\HH^i(Q^\bullet) = \HH^i(P^\bullet) = 0 $ 
if $i < n_1$, we may replace $Q^{n_1}$ by $Q^{n_1}/\BB^{n_1}(Q^\bullet)$ and $Q^i$ by $0$ for $i < n_1$.
Let $N=\prod_{i=n_1}^{0}N(i)$ and let $N'=\sum_{i=n_1}^{0}N'(i)$ and
define $J=B\cdot(w_2^N - 1)^{N'}$.
Then $J$ is a closed two-sided ideal which
lies inside $J_{n_1}$. Since $J_{n_1}$  annihilates $Q^i$ 
for all $i$, step 1 follows.

\medbreak

\noindent\textit{Step $2$}: 
We can assume that the complex $Q^\bullet$ from
step 1 has the property that all of the $Q^i$ are finitely generated 
as abstract $A$-modules.

\medbreak

\noindent\textit{Proof of Step $2$.}
Let $J$ be the ideal from step 1. 
By Remark \ref{rem:hartshorne}, $Q^\bullet$ is isomorphic in $D^-(B/J)$
to a complex ${Q'}^\bullet$ whose terms are zero in positive degrees
and free and finitely generated as abstract $B/J$-modules
in non-positive degrees. Let ${Q''}^\bullet$ be 
the complex obtained from ${Q'}^\bullet$ by replacing ${Q'}^{n_1}$ 
by ${Q'}^{n_1}/\BB^{n_1}({Q'}^\bullet)$ and ${Q'}^i$ by $0$ for $i < n_1$. 
By replacing $Q^\bullet$ by ${Q''}^\bullet$, we can assume that all of the terms 
$Q^i$ are finitely generated as abstract $B/J$-modules.

Suppose by induction that $n_0$ is an integer such that $Q^{i}$ is finitely generated 
as an abstract $A$-module for all 
integers $i < n_0$.  This hypothesis certainly holds when $n_0 = n_1$, 
since $Q^i = 0$ for $i < n_1$.     
Since $\BB^{n_0}(Q^\bullet)=\mathrm{Image}(Q^{n_0-1} \to Q^{n_0})$ 
and $\HH^{n_0}(Q^\bullet)$ are finitely generated as abstract $A$-modules, also
$\ZZ^{n_0}(Q^\bullet) = \mathrm{Ker}(Q^{n_0} \to Q^{n_0+1})$ is finitely 
generated as an abstract $A$-module.
We apply Proposition \ref{prop:secondprop} to the modules $M = Q^{n_0}$ and 
$T =\ZZ^{n_0}(Q^\bullet) $, where, as arranged above, 
$Q^{n_0}$ is finitely generated as an abstract $B/J$-module.
This shows that there is a pseudocompact
$B/J$-submodule $M'$ of $M$ such that $M' \cap \ZZ^{n_0}(Q^\bullet) = \{0\}$
and $Q^{n_0}/M'$ is finitely generated as an abstract $A$-module.  
The restriction of the differential $\delta^{n_0}:Q^{n_0} \to Q^{n_0 + 1}$
to $M'$ is therefore injective.  This implies that we have an exact sequence 
in $C^-(B/J)$
$$0 \to Q_2^\bullet \to Q^\bullet  \to Q_1^\bullet \to 0$$
in which $Q_2^\bullet$ consists of the two-term complex $M' \to \delta^{n_0}(M')$ in 
degrees $n_0$ and $n_0 + 1$, and the morphism $Q_2^\bullet \to Q^\bullet$ results 
from the natural inclusions of these terms into $Q^{n_0}$ and $Q^{n_0 + 1}$, 
respectively.  Since $Q_2^\bullet$ is acyclic, $Q^\bullet \to Q_1^\bullet$ is a 
quasi-isomorphism.  The term $Q_1^i$ is $Q^i$ if $i < n_0$,
and if $i = n_0$ then $Q_1^{n_0}=Q^{n_0}/M'$ which is finitely generated 
as an abstract $A$-module.  One now replaces $Q^\bullet$ by $Q_1^\bullet$ and 
continues by ascending induction on $n_0$. Hence step 2 follows.
 
\medbreak

\noindent \textit{Step $3$}: 
The complex $Q^\bullet$ from step 2 is isomorphic in $D^-(B)$ to 
a complex $L^\bullet$ such that $L^i=0$ for $i>0$ and $L^i$ is free and
finitely generated as an abstract $A$-module for $i\le 0$.

\medbreak

\noindent\textit{Proof of Step $3$.}
We construct $L^\bullet$ using
Proposition \ref{prop:secondstep} together with a modification of the procedure described in
\cite[III Lemma 12.3]{Hartshorne}.  

If $n\le 0$ is an integer, let $Q^{>n}$ be the truncation of $Q^\bullet$ which results by 
setting to $0$ all terms in degrees $\le n$.  
Suppose by induction that $L^{>n}$ is a complex in
$D^-(B)$ with the following properties.  The terms of $L^{>n}$ are free
and finitely generated as abstract $A$-modules and these terms are $0$ in 
dimensions $\le n$ and in dimensions $>0$. 
Moreover, there is a morphism $\pi^{> n}:L^{> n} \to Q^{>n}$ 
in $C^-(B)$
which induces isomorphisms $\HH^i(L^{>n} ) \to \HH^i(Q^\bullet)$
for $i>n+1$ and for which the induced map $\ZZ^{n+1}(L^{>n}) \to 
\HH^{n+1}(Q^\bullet)$
is surjective.  We can certainly construct such an $L^{>n}$ for $n  = 0$ since $Q^i = 0$ for $i > 0$.

The pseudocompact $B$-module $\ZZ^n(Q^\bullet)$ is finitely generated 
as an abstract $A$-module since it is a submodule of $Q^n$ and $A$ is Noetherian.
Therefore, by Proposition \ref{prop:secondstep}, there exists a 
pseudocompact $B$-module $L_1^n$
that is free and finitely generated as an abstract $A$-module together with a surjection
$\tau_1:L_1^n \to \ZZ^n(Q^\bullet)$.  
Let $\pi^{n+1}:L^{n+1} \to Q^{n+1}$ be the morphism
defined by $\pi^{>n}$. 
Define $M$ to be the pullback:
\begin{equation}
\label{eq:pulldiag}
\xymatrix{
M\ar[r]\ar[d]& (\pi^{n+1})^{-1}(\BB^{n+1}(Q^\bullet))\ar^{\pi^{n+1}}[d]  \\
 Q^n \ar^(.4){\delta^n}[r] &  \BB^{n+1}(Q^\bullet) .}
\end{equation}
Because $(\pi^{n+1})^{-1}(\BB^{n+1}(Q^\bullet))$ is contained in $L^{n+1}$, it
is finitely generated as an abstract
$A$-module.  Since $Q^n$ is also finitely generated as an abstract $A$-module,
it follows that the pseudocompact
$B$-module $M$ is finitely generated as an abstract $A$-module.
Note that the top horizontal morphism in (\ref{eq:pulldiag}) is surjective
because the lower horizontal morphism is surjective.  

By Proposition \ref{prop:secondstep}, there exists a pseudocompact
$B$-module $L_2^n$ that is free and finitely generated as an abstract $A$-module 
together with a surjection $\tau_2:L_2^n \to M$ of pseudocompact $B$-modules.  
This and (\ref{eq:pulldiag}) lead to a diagram of the following kind:
\begin{equation}
\label{eq:pulldiag2}
\xymatrix{
L^n = L^n_1 \oplus L^n_2\ar[r]^(.6){d^n}\ar[d]_{\pi^n}& L^{n+1}\ar[d]^{\pi^{n+1}}  \\
 Q^n \ar[r]^{\delta^n} &  Q^{n+1}.}
\end{equation}
Here the restriction of $d^n:L^n_1 \oplus L^n_2 \to L^{n+1}$ to $L^n_1$ is trivial, 
and the restriction
of $d^n$ to $L^n_2$ is the composition of the surjection $\tau_2:L^n_2 \to M$
with the morphism $M\to \pi_{n+1}^{-1}(\BB^{n+1}(Q^\bullet))$ in the top row
of (\ref{eq:pulldiag}) followed by the inclusion of $\pi_{n+1}^{-1}(\BB^{n+1}(Q^\bullet))$ 
into $L^{n+1}$.  The restriction of the left downward morphism
$\pi^n:L^n = L^n_1 \oplus L^n_2 \to Q^n$ to $L^n_1$ is the composition of 
$\tau_1:L_1^n \to \ZZ^n(Q^\bullet)$ with the inclusion of $\ZZ^n(Q^\bullet)$
into $Q^n$, and the restriction of this morphism to $L^n_2$ results
from the surjection $\tau_2:L^n_2 \to M$ followed by the left downward 
morphism in (\ref{eq:pulldiag}).

By construction, the diagram (\ref{eq:pulldiag2}) is commutative, and gives
a morphism $\pi^{>(n-1)}:L^{>(n-1)} \to Q^{>(n-1)}$ in $C^-(B)$.  We assumed
that the morphism $\ZZ^{n+1}(L^\bullet) \to \HH^{n+1}(Q^\bullet)$, which is induced
by $\pi^{>n}$, is surjective.  Since the top
horizontal morphism in (\ref{eq:pulldiag}) is surjective, the image of
$d^n:L^n \to L^{n+1}$ is $(\pi^{n+1})^{-1}(B^{n+1}(Q^\bullet)) \subset L^{n+1}$.
It follows that $\pi^{>(n-1)}:L^{>(n-1)} \to Q^{>(n-1)}$ induces an isomorphism
$$\HH^{n+1}(L^{>(n-1)}) \to \HH^{n+1}(Q^\bullet).$$
Because $L_1^n \subset \ZZ^n(L^{>(n-1)})$,
we also have that $\pi^n:\ZZ^n(L^{>(n-1)}) \to \ZZ^n(Q^\bullet)$
is surjective.  So since $L^n$ is free and finitely generated as an abstract $A$-module,
we conclude by induction that we can construct a bounded above
complex $L^\bullet$ in $D^-(B)$ whose terms are free and finitely generated as 
abstract $A$-modules
together with a quasi-isomorphism  $L^\bullet \to Q^\bullet$ in $C^-(B)$.
This completes the proof of step 3.
 
 \medbreak
 
Since $L^\bullet$ from step 3 is isomorphic to $P^\bullet$ in $D^-(B)$,
$L^\bullet$ satisfies hypotheses (ii) and (iii) of Theorem \ref{thm:derivedresult}.
By Definition \ref{def:quasilifts}(a), this implies that
$L^\bullet$ has finite pseudocompact $A$-tor dimension at $n_1$.
Since all the terms of $L^\bullet$ are topologically free by Remark \ref{rem:newandimproved}(v),
it follows by Remark \ref{rem:dumbdumb} that the bounded complex 
$C^\bullet$ that is obtained from $L^\bullet$ by replacing $L^{n_1}$ 
by $L^{n_1}/\BB^{n_1}(L^\bullet)$ and $L^i$ by $0$ for $i < n_1$,
is quasi-isomorphic to $L^\bullet$ and has topologically free
pseudocompact terms over $A$. By Remark \ref{rem:newandimproved}(v) and step 3, 
this implies that all terms of $C^\bullet$ are free and finitely generated as abstract
$A$-modules.

Because of Corollary \ref{cor:reductioabsurdum},
this completes the proof of Theorem \ref{thm:mainthm},
assuming Propositions \ref{prop:zerostep} - 
\ref{prop:secondstep} and assuming $r=1$ if $G$ is as in case B.
We will prove these propositions in \S\ref{ss:zerostep} - \S\ref{ss:secondstep}
and discuss the case $r>1$ for $G$ as in case B in \S\ref{ss:generalR}.


\subsection{Proof of Proposition \ref{prop:zerostep}}
\label{ss:zerostep}

Suppose $\Delta$ is as in Lemma \ref{lem:efissocute}(iii) and $r=1$. 
Write $w_2$ instead of $w_{2,1}$, and 
let $J=B\cdot (w_2^N-1)^{N'}$ be as in the 
statement of Proposition \ref{prop:zerostep}.
The key to proving this proposition is to
uniquely express each element in $B=A[[\Delta]]$
by a unique convergent power series as in Lemma \ref{lem:represent} below.

We first note that the left ideal $J=B\cdot (w_2^N-1)^{N'}$ 
is a two-sided ideal in $B$, since 
\begin{equation}
\label{eq:commute}
(w_2^N-1)^{N'}\,\overline{\Phi}^{\,-1} = \overline{\Phi} \,(w_2^{\ell^f N}-1)^{N'}
=\overline{\Phi} \left(\sum_{i=0}^{\ell^f -1} w_2^{i N}\right)^{N'} (w_2^N-1)^{N'}.
\end{equation}
Suppose that in the description of $\Delta_0$
in Lemma \ref{lem:efissocute}(iii), the finite cyclic $p'$-group $\tilde{\Delta}_0$ of order 
$d$ is generated by $\sigma\in\Delta$.

\begin{lemma}
\label{lem:represent}
Write $N = p^s t$ where $s \ge 0$ and $t$ is prime to $p $. Then 
$w_2^N - 1 = (w_2^{p^s} - 1)\cdot v$
where $v$ is a unit of $B$ commuting with $w_2$,
so $J=B \cdot (w_2^{p^s} - 1)^{N'}$.
Every element $f$ of $B$ can be written in a 
unique way as a convergent power series
\begin{equation}
\label{eq:theform}
f=\sum \; z_{u,a,\xi,b,c} \;\sigma^u\; (w_1 - 1)^a \;\xi \;
w_2^b\; (w_2^{p^s} - 1)^c,
\end{equation}
in which the sum ranges over all tuples $(u,a,\xi,b,c)$
with $0\le u\le d-1$, $a\ge 0$, $\xi\in\tilde{\Delta}_1$, $0\le b\le p^s - 1$
and $c \ge 0$, and each $z_{u,a,\xi,b,c}$ lies in $A$.
Moreover, any choice of $z_{u,a,\xi,b,c}\in A$ 
defines an element $f\in B$.
\end{lemma}

\begin{proof}  A cofinal system of closed normal finite index 
subgroups of $\Delta$ is given by the groups 
$H(m,m')$ that are topologically generated 
by $w_2^{p^{s+m}}$ and $w_1^{p^{m'}}$, 
where $m\ge 0$ is arbitrary and $m'$ is chosen so that 
$\ell^{fdp^{m'}} \equiv 1$ mod $p^{s+m}$ and the order of the automorphism of 
the finite group $\tilde{\Delta}_1$ induced by the pro-$p$ element $w_1$ divides 
$p^{m'}$.
Note that these requirements on $m'$ ensure that each $H(m,m')$ is
normal in $\Delta$.
Define $\Gamma(m,m') = \Delta/H(m,m')$.

In $B$, we have
$$w_2^N - 1 = w_2^{p^s t}-1 = (w_2^{p^s} - 1)\cdot v,$$
where $v=1 + w_2^{p^s} + \cdots + w_2^{p^s \cdot (t-1)}$ is congruent to
$t$ mod the two-sided ideal $B\cdot (w_2-1)$. Since $t\not\equiv 0\mod p$,
$v$ has invertible image in $B/\left(B\cdot (w_2-1)\cap pB\right)$. 
Since the two-sided ideal
$B\cdot (w_2-1)\cap pB$ has nilpotent image in $A'[\Gamma(m,m')]$
for all discrete Artinian quotients $A'$ of $A$,
this implies that $v$ is a unit in $B$.

Since $B=A[[\Delta]]$ is the projective limit of the quotient rings $A'[\Gamma(m,m')]$,
as $A'$ ranges over all discrete Artinian quotients of $A$ and $(m,m')$ ranges over 
all pairs of integers satisfying the above conditions, it follows that every
$f\in B$ can be written in a unique way as a power series as in (\ref{eq:theform})
and every such power series converges to an element in $B$.
 \end{proof}
 
\begin{rem}
\label{rem:alternate}
By a similar argument, every element $f$ of $B$ can be written in a 
unique way as a convergent power series
\begin{equation}
\label{eq:theformalt}
f = \sum\; 
\omega_{u,a,\xi,b,c}\;(w_2^{p^s} - 1)^c \;w_2^b\;\,
\xi\; (w_1 - 1)^a\; \sigma^u,
\end{equation}
in which the sum ranges over all tuples $(u,a,\xi,b,c)$
with $0\le u\le d-1$, $a\ge 0$, $\xi\in\tilde{\Delta}_1$, $0\le b\le p^s - 1$
and $c \ge 0$, and each $\omega_{a,b,\xi,c,u}$ lies in $A$. 
Moreover, any choice of  $\omega_{a,b,\xi,c,u} \in A$
defines and element in $B$.
\end{rem} 

To prove Proposition \ref{prop:zerostep}, let $J=B\cdot (w_2^N-1)^{N'}$.
By $(\ref{eq:commute})$, $J$ is a two-sided ideal in $B$.
By Remark \ref{rem:useful}(i),  $J$ is closed in $B$,
which implies that $\overline{B}=B/J$ is a pseudocompact $A$-algebra. 
By using Lemma \ref{lem:represent} (resp. Remark \ref{rem:alternate}), we see
that right (resp. left) multiplication with $(w_2^{p^s} - 1)^{N'}$
is an injective homomorphism $B \to B$.
This shows that $J$ is an abstractly free rank one left (resp. right)
$B$-module. By Remark \ref{rem:useful}(iii), 
this proves Proposition \ref{prop:zerostep}.

\subsection{Proof of Proposition \ref{prop:firststep}}
\label{ss:firststep}

Suppose $\Delta$ is as in Lemma \ref{lem:efissocute}(iii) and $r=1$.
Write $w_2$ instead of $w_{2,1}$. 
We will prove Proposition \ref{prop:firststep} by proving Lemmas \ref{lem:step11} and
\ref{lem:step12} below, which enable us to essentially reduce to the case when
$A$ is a field.

\begin{lemma}
\label{lem:step11}
Let $L$ be a field and
let $M$ be a pseudocompact $L[[\Delta]]$-module that is finite dimensional 
as $L$-vector space.
There exist positive integers $N,N'$ which are bounded functions of 
$\mathrm{dim}_L\,M$ such that $(w_2^N-1)^{N'}\cdot M=\{0\}$. 
\end{lemma}

\begin{proof}
The action of $w_2$ on $M$ defines an automorphism  in 
$\mathrm{Aut}_L(M)\subset \mathrm{End}_L(M)$.
This implies that there is a monic polynomial $h(x)\in L[x]$ of degree 
less than or equal to $(\mathrm{dim}_L\,M)^2$
such that $h(w_2)\cdot M=0$. Since $w_2$ 
is a unit, we can assume that $h(x)$ is not divisible by $x$. Since 
$w_1h(w_2)w_1^{-1}=h(w_2^{\ell^{fd}})$, it follows that $h(w_2^{\ell^{fd}})$ 
also annihilates $M$. 

Let $I$ be the 
ideal of $L[x]$ that is (abstractly) generated by $\{h(x^{\ell^{fdn}})\;|\; n\ge 0\}$. 
Then $f(w_2)\cdot M=0$ for all $f(x)\in I$.
Since $L[x]$ is a principal ideal 
domain, $I$ is generated by a single polynomial $d(x)\in L[x]$. Moreover, since 
$x$ does not divide $h(x)$, $x$ also does not 
divide $d(x)$. Because $d(x^{\ell^{fd}})\in I$, $d(x)$ divides $d(x^{\ell^{fd}})$. 
This means that if  $\{\rho_1,\ldots,\rho_m\}$ are
the roots of $d(x)$, then for each $1\le i\le m$, $\{\rho_i^{\ell^{fdn}}\;|\;n\ge 0\}$ is 
contained in $\{\rho_1,\ldots,\rho_m\}$.
Note that $m\le \mathrm{deg}\,d(x)\le \mathrm{deg}\, h(x)\le (\mathrm{dim}_L\,M)^2$. 
This implies that there exists a positive integer $s$, which is a bounded function of 
$\mathrm{dim}_L\,M$, such that each $\rho_i$ is a root of unity of finite 
order bounded by $\ell^{fds}$. Thus $d(x)$ divides a polynomial of the form 
$(x^N-1)^{N'}$ where $N,N'$ are bounded functions of $\mathrm{dim}_L\,M$.
\end{proof}

\begin{cor}
\label{cor:step11}
Suppose 
$M$ is
a pseudocompact $A[[\Delta]]$-module that is finitely generated as an abstract 
$A$-module.
There exist positive  integers $N,N''$ that are bounded functions of the number of 
abstract
generators of $M$ over $A$ such that $(w_2^N-1)^{N''}$ annihilates 
$k(\mathfrak{p})\hat{\otimes}_AM$ for all prime ideals $\mathfrak{p}$
of $A$, where $k(\mathfrak{p})$ denotes the residue field of $\mathfrak{p}$.
\end{cor}

\begin{proof}
Note that $\mathrm{dim}_{k(\mathfrak{p})}(k(\mathfrak{p})\otimes_A M)$ is less than or 
equal to the number of generators of $M$ as an abstract $A$-module. Hence we can 
use Lemma \ref{lem:step11} with $k(\mathfrak{p})$ for $L$ and 
$k(\mathfrak{p})\hat{\otimes}_A M$ for $M$.
\end{proof}

\begin{lemma}
\label{lem:step12}
Let $M$ be as in Corollary $\ref{cor:step11}$. 
Suppose $f\in\mathrm{End}_A(M)$
 and that for all prime ideals $\mathfrak{p}$ of $A$ we have
\begin{equation}
\label{eq:step12}
f(M)_{\mathfrak{p}} \subseteq \mathfrak{p}\cdot M_{\mathfrak{p}}
\end{equation}
where the subscript $\mathfrak{p}$ means localization at the prime ideal 
$\mathfrak{p}$. Then $f$ is nilpotent.
\end{lemma}

\begin{proof}
Let first $\mathfrak{p}$ be a prime ideal of $A$ of codimension $0$. Then  
$\mathrm{dim}\, A_{\mathfrak{p}}=0$ so that $A_{\mathfrak{p}}$ is Artinian. 
Because $f(M)$ is finitely generated as an abstract
$A$-module, this implies that 
$f(M)_{\mathfrak{p}}$ is an Artinian $A_{\mathfrak{p}}$-module. Since by assumption,
$f(M)_{\mathfrak{p}} \subseteq \mathfrak{p}\cdot M_{\mathfrak{p}}$, we obtain for all 
positive integers $n$ that $f^n(M)_{\mathfrak{p}} \subseteq \mathfrak{p}^n\cdot 
M_{\mathfrak{p}}$. Thus there is a positive integer $n(\mathfrak{p})$ with 
$f^{n(\mathfrak{p})}(M)_{\mathfrak{p}}=0$.
Since $A$ is Noetherian, there are only finitely many prime ideals of $A$ of codimension $0$.
Hence there is a positive integer $n_0$ such that $f^{n_0}(M)_{\mathfrak{p}}=0$ for all 
prime ideals $\mathfrak{p}$ of $A$ of codimension $0$.

Now let $t\ge 1$, and suppose by induction that there is an integer $n_{t-1}$ such that 
$f^{n_{t-1}}(M)_{\mathfrak{q}}=0$ for all prime ideals $\mathfrak{q}$ of $A$ of  codimension 
at most $t-1$.
In particular, for each prime ideal $\mathfrak{q}$ of $A$ of codimension at most $t-1$ there exists
an element $b(\mathfrak{q})\in A$ such that $b(\mathfrak{q})\not\in\mathfrak{q}$ and
$b(\mathfrak{q})\cdot f^{n_{t-1}}(M)=0$. Let $I_{t-1}$ be the ideal of $A$ that is
abstractly generated by all elements $b(\mathfrak{q})$ as $\mathfrak{q}$ ranges over all 
prime ideals  of $A$ of codimension at most $t-1$.

Let $\mathfrak{p}$ be a prime ideal of $A$ of codimension $t$. 
Using that the non-zero prime ideals of 
$(A/I_{t-1})_{\mathfrak{p}}\cong A_{\mathfrak{p}}/I_{t-1} A_{\mathfrak{p}}$ 
correspond to the prime ideals of $A_{\mathfrak{p}}$ 
containing $I_{t-1} A_{\mathfrak{p}}$ and that the prime ideals of 
$A_{\mathfrak{p}}$ correspond 
to the prime ideals of $A$ contained in $\mathfrak{p}$, one shows that
$(A/I_{t-1})_{\mathfrak{p}}$ has dimension $0$.
Since $f^{n_{t-1}}(M)$ is finitely generated as an abstract $(A/I_{t-1})$-module, 
one shows, similarly to the first paragraph of this proof, that
there is a positive integer $n(\mathfrak{p})$ with 
$f^{n(\mathfrak{p})}(M)_{\mathfrak{p}}=0$.
Since $f^{n_{t-1}}(M)$ is supported in codimension $t$, 
it follows that
there is a positive  integer $n_t$ such that $f^{n_t}(M)_{\mathfrak{p}}=0$ for all prime 
ideals $\mathfrak{p}$ of $A$ of codimension at most $t$.

Since $A$ is Noetherian, the codimensions of all prime ideals of $A$ are bounded above by
a fixed non-negative integer. Hence we obtain that there is a positive integer $n$ such that
$f^{n}(M)_{\mathfrak{p}}=0$ for all prime ideals $\mathfrak{p}$ of $A$,
which implies $f^n(M) = 0$.
\end{proof}

\medbreak

To prove Proposition \ref{prop:firststep}, let $M$ be
a pseudocompact $B$-module that is finitely generated as an abstract 
$A$-module. By Corollary \ref{cor:step11}, there exist positive  integers
$N,N''$ that are bounded functions of the number of abstract
generators of $M$ such that $(w_2^N-1)^{N''}$ annihilates 
$k(\mathfrak{p})\hat{\otimes}_AM$ for all prime ideals  $\mathfrak{p}$ of $A$.
Letting $f$ be the endomorphism of $M$ defined by the action of 
$(w_2^N-1)^{N''}$ on $M$,
this is equivalent to condition $(\ref{eq:step12}$)
for all prime ideals $\mathfrak{p}$ of $M$.
Hence Lemma \ref{lem:step12} implies that $f$ is nilpotent, i.e.
there exists an integer $N'$ such that $(w_2^N-1)^{N'}$ annihilates $M$.
This proves Proposition \ref{prop:firststep}.


 \subsection{Proof of Proposition \ref{prop:secondprop}}
\label{ss:secondprop}

Let $J$, $\Lambda$, $M$ and $T$ be as in the statement of Proposition
\ref{prop:secondprop}. 
The main idea for proving this proposition is to use the Artin-Rees Lemma to construct
the almost complement $M'$ for $T$. This works directly if $\Delta$ is
abelian. If $\Delta$ is as in Lemma \ref{lem:efissocute}(iii) and $r=1$, we first use the
Artin-Rees Lemma 
for the case when the exponent $N'$  in the definition
of the ideal $J$ is equal to 1 and then use an inductive argument 
in the general case. Note that $J$ is a closed two-sided ideal of $B$ by Proposition
\ref{prop:zerostep}, so $\Lambda=B/J$ is a pseudocompact $A$-algebra.

\medbreak

Suppose first that $\Delta$ is as in Lemma \ref{lem:efissocute}(ii), i.e. $J=0$
and $\Lambda=B=A[[\Delta]]$ is commutative. 
For $1\le j\le s$, the action of $w_{1,j}$ on $T$ defines an automorphism in
$\mathrm{Aut}_A(T)\subset \mathrm{End}_A(T)$. Since $T$ is finitely generated
as an abstract $A$-module, it follows that the same is true for $\mathrm{End}_A(T)$.
Hence there exists a monic polynomial $F_j(x)\in A[x]$ such that 
$F_j(w_{1,j})$ annihilates $T$. Let $I$ be the ideal in the commutative
Noetherian ring $\Lambda$ that is abstractly
generated by $F_j(w_{1,j})$ for $1\le j\le s$. 
By the Artin-Rees Lemma, there is an 
integer $q >>0$ such that $T \cap (I^{q+1} \cdot M) =I\cdot (T \cap (I^q \cdot M))$.
However, $I$ annihilates $T$ by construction, so we
conclude that $T \cap (I^{q+1}\cdot M) = \{0\}$.  
Since $\Lambda$ is commutative and $I^{q+1}$ is abstractly finitely generated,
it follows that
$I^{q+1} \cdot M$ is
a pseudocompact $\Lambda$-submodule of $M$ by Remark \ref{rem:useful}(i).
The quotient $M/(I^{q+1} \cdot M)$ is finitely generated 
as an abstract module for the ring $\Lambda/  I^{q+1}$, and this ring is finitely
generated as an abstract $A$-module, since $I$ contains a monic polynomial in 
$w_{1,j}$ for each $1\le j\le s$ and $\Delta/\langle w_{1,1},\ldots,w_{1,s}\rangle$
is finite. Hence $M/(I^{q+1} \cdot M)$ is finitely generated as an abstract
$A$-module and Proposition \ref{prop:secondprop} is proved
if $\Delta$ is as in Lemma \ref{lem:efissocute}(ii).

\medbreak

Suppose now that $\Delta$ is as in Lemma \ref{lem:efissocute}(iii) and $r=1$. 
Write $w_2$ instead of $w_{2,1}$. Then $J=B\cdot (w_2^N-1)^{N'}$
for positive integers $N,N'$ and $\Lambda=B/J$.
Suppose first that $N'=1$. Then $J=B\cdot (w_2^{p^s}-1)$
by Lemma \ref{lem:represent}, and hence $\Lambda=B/J=A[[\overline{\Delta}]]$, where
$\overline{\Delta}$ is the quotient of $\Delta$ by the closed normal
subgroup that is topologically generated by 
$w_2^{p^s}$.   The conjugation action of $w_1$ on  the 
finite normal abelian subgroup
$\left(\langle w_2 \rangle\times \tilde{\Delta}_1\right) / \langle w_2^{p^s}\rangle$ 
of $\overline{\Delta}$ gives an automorphism of 
finite order. Thus $w_1^z$ is in the center of $\overline{\Delta}$, and of 
$\Lambda$, if $z \ge 1$ is sufficiently divisible.  
Similarly to the case when $\Delta$ is as in Lemma \ref{lem:efissocute}(ii),
one finds an ideal $I$ in $A[[\langle  w_1^z \rangle]]$ that is abstractly
generated by a monic polynomial in $w_1^z$ such that  $T \cap (I^{q+1}\cdot M) = 
\{0\}$ for an integer $q> > 0$. 
Since $M$ is a pseudocompact $\Lambda$-module and
$I^{q+1}$ is generated by a single element that 
lies in the center of $\Lambda$, it follows that
$I^{q+1} \cdot M$ is
a pseudocompact $\Lambda$-submodule of $M$ by Remark \ref{rem:useful}(i).
The quotient $M/(I^{q+1} \cdot M)$ is finitely
generated as an abstract module for the ring $\Lambda/ I^{q+1}$. This ring is finitely
generated as an abstract $A$-module, since $I$ contains a monic polynomial in 
$w_1$, since $w_2^{p^s}=1$ in $\Lambda$, and since $\tilde{\Delta}_0$ and 
$\tilde{\Delta}_1$ are finite. 
So $M/(I^{q+1} \cdot M)$ is finitely generated as an abstract
$A$-module and Proposition \ref{prop:secondprop} is proved 
if $N' = 1$.

We now suppose that $N'  \ge 1$ are arbitrary.  
In this case, we break the proof into several steps given by Lemma \ref{lem:fisit},
Corollary \ref{cor:neoth} and Lemma \ref{lem:itsfinite} below.
For simplicity, let 
$\epsilon = (w_2^{p^s} -  1)$ so that $J = B\cdot \epsilon^{N'}$.  
For $m  \ge 1$, define
\begin{equation}
\label{eq:mmm}
M(\epsilon^m) = \{\alpha \in M\;|\; \epsilon^m
 \cdot \alpha = 0\}.
 \end{equation}
Since by Proposition \ref{prop:zerostep},
$\Lambda\cdot \epsilon^m$ is a two-sided ideal of $\Lambda$,
it follows that $M(\epsilon^m)$ is a 
pseudocompact $\Lambda$-submodule of $M$.

\begin{lemma}
\label{lem:fisit}  
The module $M$ is a left Noetherian $\Lambda$-module.  
If $M(\epsilon)$ is not finitely generated as an abstract $A$-module, 
then there exists a non-zero pseudocompact $\Lambda$-submodule 
$Y$ of $M(\epsilon)$ such that $T \cap Y = 0$. 
\end{lemma}

\begin{proof} 
Since $M$ is finitely generated as an abstract $\Lambda$-module, where 
$\Lambda=B/J$, and 
$B$ is left Noetherian, $M$ must be a left Noetherian $\Lambda$-module.
By $(\ref{eq:mmm})$, $M(\epsilon)$ is annihilated by 
$\epsilon$. Thus
$M(\epsilon)$ is a pseudocompact $\Lambda_1$-module, where
$$\Lambda_1 = \Lambda/\Lambda\epsilon = 
B/B\cdot (w_2^{p^s} - 1),$$ and 
$T_1 = T \cap M(\epsilon)$ is a pseudocompact 
$\Lambda_1$-submodule of $M(\epsilon)$.
Because $T$ is finitely generated as an abstract $A$-module,
$T_1$ is also finitely generated as an abstract $A$-module.  
By what we proved in the case when
$N'= 1$, we can therefore conclude that there is a 
pseudocompact $\Lambda$-submodule 
$Y$ of $M(\epsilon)$ such that $T \cap Y = T_1 \cap Y = \{0\}$ 
and $M(\epsilon)/Y$ is finitely generated as an abstract $A$-module.  
If $M(\epsilon)$ is not finitely generated as an abstract $A$-module,
this forces $Y$ to be non-zero.
\end{proof}
 
\begin{cor} 
\label{cor:neoth}
There is a pseudocompact $\Lambda$-submodule $M'$ of $M$ such that 
$T\cap M' = 0$ and $(M/M')(\epsilon)$ is finitely generated as
an abstract $A$-module.  
\end{cor} 
 
\begin{proof} 
Suppose we have constructed for some integer
$n \ge 0$ a strictly increasing sequence of pseudocompact $\Lambda$-submodules 
$M_0 \subset M_1 \subset \cdots \subset M_n$ of $M$ such
that $M_0 = \{0\}$ and $T \cap M_n = \{0\}$.   If $(M/M_n)(\epsilon)$ 
is finitely generated as an abstract $A$-module, 
then we let $M' = M_n$ and we are done.  Otherwise, 
observe that $T$ injects into $M/M_n$.  We can apply Lemma \ref{lem:fisit} to this 
inclusion and to the module $M/M_n$ to conclude that there  a non-zero 
pseudocompact $\Lambda$-submodule 
$Y$ of $(M/M_n)(\epsilon)$ such that $T \cap Y = 0$.  The inverse 
image of $Y$ in $M$ is a pseudocompact $\Lambda$-submodule 
$M_{n+1}$ which properly contains $M_n$ and for 
which $T \cap M_{n+1} = \{0\}$. Since $M$ is left Noetherian by Lemma \ref{lem:fisit}, 
the process stops at some $n$, 
meaning that $(M/M_n)(\epsilon)$ is finitely generated as
an abstract $A$-module, 
and we can let $M' = M_n$.  
 \end{proof}

\begin{lemma}
\label{lem:itsfinite}
If $M(\epsilon)$ is finitely generated as an abstract $A$-module, 
then $M$ is finitely generated as an abstract $A$-module. 
\end{lemma}

\begin{proof}
We show this by proving by increasing induction on 
$m$ that $M(\epsilon^m)$ is finitely generated as an abstract  $A$-module for all 
$m \ge 1$.  
When $m  = 1$, this statement holds by assumption.
Suppose now that it is true for some $m\ge 1$. We have an exact sequence
of  $A$-modules 
$$0 \to M(\epsilon) \to 
M(\epsilon^{m+1}) \to   
M(\epsilon^m)$$
in which the $A$-linear map $M(\epsilon^{m+1}) \to M(\epsilon^m)$ is multiplication 
by $\epsilon$.
Since $M(\epsilon)$ 
and $M(\epsilon^m)$
are finitely generated as abstract $A$-modules by induction, this proves that 
$M(\epsilon^{m+1})$ is
finitely generated as an abstract $A$-module.  
Since $\epsilon^{N'} = 0$ in $\Lambda$, 
we conclude that  $M(\epsilon^{N'}) = M$ is finitely 
generated as an abstract $A$-module.
 \end{proof}


\subsection{Proof of Proposition \ref{prop:bounded}}
\label{ss:bounded}

As in the statement of Proposition \ref{prop:bounded}, let
$\Omega$ be a pseudocompact left Noetherian ring and let 
$P^\bullet$ be a complex in $D^-(\Omega)$ whose terms
$P^i$ are free and finitely generated as abstract $\Omega$-modules 
such that $P^i = 0$ for $i > 0$.  For $i \le 0$, assume that $I_i$ is a closed two-sided 
ideal in $\Omega$ that annihilates $\HH^i(P^\bullet)$ such that
$J_i=I_i \cdot I_{i+1} \cdots I_1 \cdot I_0$ is free and finitely generated 
as an abstract left $\Omega$-module.
We need to prove that $P^\bullet$ is isomorphic in $D^-(\Omega)$ to a complex 
$Q^\bullet$ such that $Q^i = 0$ for $i > 0$ and $Q^i$ is annihilated by $J_i$ 
for $i \le 0$. We will prove this by constructing
$Q^\bullet$ inductively from right to left.

Let $j\le 0$ be an integer. Suppose by induction that $Q^{>j}$ is a complex
which is isomorphic to $P^\bullet$ in $D^-(\Omega)$ with the following properties.
The terms $Q^i$ are zero for $i>0$ and free and finitely generated as abstract
$\Omega$-modules for $i\le j$. Also, for $j+1\le i\le 0$, $Q^i$ is annihilated by 
$J_i$
and is finitely generated as an abstract 
$\Omega$-module. We can certainly construct such a complex $Q^{>j}$ when $j=0$ 
since then we can simply let $Q^{>0} = P^\bullet$.

\medbreak

\noindent\textit{Claim $1$}: 
The complex $Q^{>j}$ is isomorphic in $D^-(\Omega)$ to a complex
$Q_1^\bullet$ such that $Q_1^i=Q^i$ for $i>j$, $Q_1^j$ is 
annihilated by $J_j$ and $Q_1^i$ is finitely generated as an abstract $\Omega$-module
for $i\le j$.

\medbreak

\noindent
\textit{Proof of Claim $1$.}
The differential $\delta^j:Q^j \to Q^{j+1}$ of $Q^{>j}$ induces
an exact sequence of pseudocompact $\Omega$-modules
\begin{equation}
\label{eq:dohdoh}
\xymatrix @R.7pc{
0\ar[r]&\displaystyle\frac{\ZZ^j(Q^{>j})}{\BB^j(Q^{>j})}\ar[r]
\ar@{=}[d]& \displaystyle\frac{Q^j}{\BB^j(Q^{>j})}\ar[r]^{\overline{\delta^j}}& Q^{j+1}.\\
&\HH^j(Q^{>j})
}
\end{equation}
Since $\HH^j(Q^{>j}) = \HH^j(P^\bullet)$ has been assumed 
to be annihilated by $I_j$ and $Q^{j+1}$ is annihilated by 
$J_{j+1}=I_{j+1} I_{j+2} \cdots I_0$ by induction, (\ref{eq:dohdoh}) shows
that $Q^j/\BB^j(Q^{>j})$ is annihilated by $J_j =  I_j (I_{j+1} \cdots I_0) $.  
Hence $J_j\,Q^j$ lies in $\BB^j(Q^{>j})$ and
we obtain a short exact sequence of pseudocompact $\Omega$-modules
\begin{equation}
\label{eq:Zsplit}
0 \to \ZZ^{j-1}(Q^{>j}) \to (\delta^{j-1})^{-1}(J_j\,Q^j) \xrightarrow{\delta^{j-1}} 
J_j\,Q^j \to 0.
\end{equation}
Since, by assumption, $J_j$ is a two-sided ideal which is 
free and finitely generated as an abstract left $\Omega$-module
and since, by induction, $Q^j$ is free and finitely generated as an abstract
$\Omega$-module,
$J_j \,Q^j$ is also free and finitely generated as an 
abstract $\Omega$-module. 
By Remark \ref{rem:useful}(i),
$J_j\,Q^j$ is a pseudocompact
$\Omega$-submodule of $Q^j$.
By Remark \ref{rem:useful}(iii),  $J_j\,Q^j$ is a topologically free
pseudocompact $\Omega$-module. 
Thus there is a homomorphism $s:J_j\,Q^j \to (\delta^{j-1})^{-1}(J_j\,Q^j)$ of 
pseudocompact $\Omega$-modules such that $\delta^{j-1}\circ s$ is the identity
on $J_j\,Q^j$. In particular, $s(J_j\,Q^j)$ is a pseudocompact $\Omega$-submodule
of $(\delta^{j-1})^{-1}(J_j\,Q^j)$, and hence of $Q^{j-1}$, such that
\begin{equation}
\label{eq:notquitesure}
s(J_j\,Q^j) \cap \ZZ^{j-1}(Q^{>j}) =\{0\}.
\end{equation}
The restriction of the differential $\delta^{j-1}:Q^{j-1}\to Q^j$ to $s(J_j\,Q^j)$
is therefore injective. This implies that we have an exact sequence in $C^-(\Omega)$
$$0 \to Q_2^\bullet \to Q^{>j} \to Q_1^\bullet \to 0$$
in which $Q_2^\bullet$ is the two-term complex 
$s(J_j\,Q^j) \xrightarrow{\delta^{j-1}} J_j\,Q^j$ concentrated 
in degrees $j-1$ and $j$, and the morphism $Q^{>j} \to Q_1^\bullet$
results from the natural inclusions of these terms into $Q^{j-1}$ and $Q^j$,
respectively. Since $Q_2^\bullet$ is acyclic, $Q^{>j} \to Q_1^\bullet$ is
a quasi-isomorphism. The terms $Q_1^i$ are equal to $Q^i$ for $i>j$,
and if $i=j$ then $Q_1^j=Q^j/J_j\,Q^j$.
Moreover, since all terms of $Q^{>j}$ are finitely generated as abstract
$\Omega$-modules, the same is true for $Q_1^\bullet$.
This proves claim 1.

\medbreak

\noindent\textit{Claim $2$}: 
Let $Q_1^{\le (j-1)}$ be the truncation of $Q_1^\bullet$ which results by setting
to $0$ all terms in degrees $> j-1$. There is a quasi-isomorphism 
$\rho:L^{\le (j-1)}\to Q_1^{\le (j-1)}$ in $C^-(\Omega)$,
where $L^{\le (j-1)}$ is a bounded above complex of pseudocompact 
$\Omega$-modules that are free and finitely generated as abstract 
$\Omega$-modules and $L^i=0$ for $i>j$. Moreover, 
$\rho^{j-1}:L^{j-1}\to Q_1^{j-1}$
is surjective.

\medbreak

\noindent\textit{Proof of Claim $2$.}
This immediately follows from Remark \ref{rem:hartshorne}.

\medbreak

\noindent\textit{Claim $3$}: 
The complex $Q_1^\bullet$ from claim 1 is isomorphic in $D^-(\Omega)$ to a complex
$T^\bullet$ such that the terms $T^i$ are zero for $i>0$ and free and finitely 
generated as abstract $\Omega$-modules for $i \le j-1$. Also, for $j\le i\le 0$, $T^i$ is annihilated by $J_i$ and is finitely generated as an abstract 
$\Omega$-module.

\medbreak

\noindent\textit{Proof of Claim $3$.}
We use the complex $L^{\le (j-1)}$ and the quasi-isomorphism
$\rho:L^{\le (j-1)}\to Q_1^{\le (j-1)}$ from claim 2 to prove this.
Define $T^\bullet$ to be the complex with terms 
\begin{equation}
\label{eq:thatstheticket}
T^i=Q_1^i \quad \mbox{for $i\ge j$} \quad \mbox{and}\quad
T^i=L^i \quad \mbox{for $i\le j-1$}.
\end{equation}
Let the differentials $d_T^i$ be given by
\begin{equation}
\label{eq:thatstheticket2}
d_T^i=d_{Q_1}^i \quad \mbox{for $i\ge j$}, \quad
d_T^{j-1}=d_{Q_1}^{j-1}\circ\rho^{j-1} \quad \mbox{and}\quad
d_T^i=d_L^i \quad \mbox{for $i\le j-2$}.
\end{equation}
Define $\tau:T^\bullet\to Q_1^\bullet$ to be the map such that 
\begin{equation}
\label{eq:thatstheticket3}
\tau^i=\mbox{ identity on $Q_1^i$} \quad \mbox{for $i\ge j$} \quad \mbox{and}\quad
\tau^i=\rho^i \quad \mbox{for $i\le j-1$.}
\end{equation}
We claim that $\tau$  is a quasi-isomorphism in $C^-(\Omega)$.

It follows from the definition of $T^\bullet$ and $\tau$ in $(\ref{eq:thatstheticket})$,
$(\ref{eq:thatstheticket2})$ and $(\ref{eq:thatstheticket3})$  that
$\tau$ is a homomorphism in $C^-(\Omega)$.
Since $\tau^{j-1}=\rho^{j-1}$ is surjective by claim 2, 
it follows from $(\ref{eq:thatstheticket2})$
that 
\begin{equation}
\label{eq:into}
\BB^j(T^{\bullet}) = d_T^{j-1}(T^{j-1}) = d_{Q_1}^{j-1}(Q_1^{j-1}) = \BB^j(Q_1^\bullet).
\end{equation}
Thus the definition of $\tau:T^\bullet\to Q_1^\bullet$ in $(\ref{eq:thatstheticket3})$,
together with claim 2, 
show that $\tau$ induces an isomorphism 
$\HH^i(T^\bullet) \to \HH^i(Q_1^\bullet)$ for $i\le j-2$ and for $i\ge j$. So the only
issue is the case $i=j-1$. 
We have a commutative diagram with exact rows
\begin{equation}
\label{eq:TQ2}
\xymatrix @R.7pc{
& \HH^{j-1}(T^\bullet)\ar@{=}[d]& \HH^{j-1}(L^{\le (j-1)})\ar@{=}[d]&& {}\\
0\ar[r]& \displaystyle\frac{\ZZ^{j-1}(T^\bullet)}{\BB^{j-1}(T^\bullet)}
\ar[r]\ar[dd]_{\HH^{j-1}(\tau)}& \displaystyle\frac{T^{j-1}}{\BB^{j-1}(T^\bullet)} 
\ar[r]^(.55){\overline{d_T^{j-1}}} \ar[dd]^{\overline{\tau^{j-1}}}&B^j(T^\bullet)\ar[r]
\ar[dd]^{\tau^j}& 0\\
&&&&\\
0\ar[r]&\displaystyle\frac{\ZZ^{j-1}(Q_1^\bullet)}{\BB^{j-1}(Q_1^\bullet)}
\ar[r]\ar@{=}[d]& 
\displaystyle\frac{Q_1^{j-1}}{\BB^{j-1}(Q_1^\bullet)} 
\ar^(.55){\overline{d_{Q_1}^{j-1}}}[r]\ar@{=}[d]&B^j(Q_1^\bullet)\ar[r]& 0\\
&\HH^{j-1}(Q_1^\bullet) & \HH^{j-1}(Q_1^{\le (j-1)})&& {}
}
\end{equation}
The rightmost vertical homomorphism in $(\ref {eq:TQ2})$, 
which is induced by $\tau^j$, is an isomorphism by (\ref{eq:into}).
The middle vertical homomorphism in $(\ref {eq:TQ2})$, which is induced by
$\tau^{j-1}=\rho^{j-1}$, is equal to the isomorphism 
$\HH^{j-1}(\rho):\HH^{j-1}(L^{\le (j-1)})\to \HH^{j-1}(Q_1^{\le (j-1)})$.
So the left vertical homomorphism $\HH^{j-1}(\tau)$
in $(\ref {eq:TQ2})$ must be an isomorphism by the five lemma.
This proves claim 3.

\medbreak

It follows from claims 1 and 3 that we can let $Q^{>(j-1)}=T^\bullet$. Thus we proceed
by descending induction to construct a bounded above complex $Q^\bullet$
which is isomorphic to $P^\bullet$ in $D^-(\Omega)$ such that $Q^i = 0$ for $i > 0$ and $Q^i$ is annihilated by $J_i$ for $i \le 0$. This proves Proposition \ref{prop:bounded}.


\subsection{Proof of Proposition \ref{prop:secondstep}}
\label{ss:secondstep}

Let $M$ be a pseudocompact $B$-module that is finitely generated as an abstract
$A$-module, as in the statement of Proposition \ref{prop:secondstep}.
The key to proving this proposition is to use the Weierstrass preparation
theorem 
in a suitable power series algebra over $A$ to construct a pseudocompact
$B$-module $F$ that is free and finitely generated as an abstract $A$-module
together with a surjective homomorphism $F \to M$ of pseudocompact
$B$-modules.

\medbreak

Suppose first that $\Delta$ is as in Lemma \ref{lem:efissocute}(ii), i.e.
$\Delta = \mathbb{Z}_p^s \times Q\times Q'$. 
For $1\le j\le s$, the action of $w_{1,j}$ on $M$ defines an automorphism in
$\mathrm{Aut}_A(M)\subset \mathrm{End}_A(M)$. Since $M$ is finitely generated
as an abstract $A$-module,   the same is true for $\mathrm{End}_A(M)$.
Hence there exists a monic polynomial $g_j(x)\in A[x]$ such that 
$g_j(w_{1,j})$ annihilates $M$ for all $j$. Let $I$ be the  ideal in $B$ 
that is abstractly generated by $g_j(w_{1,j})$ for $1\le j\le s$. 
Then $I$ is 
a closed ideal of $B$ by Remark \ref{rem:useful}(i).
For $1\le j\le s$, let $x_{1,j}=w_{1,j}-1$, so that
$A[[\langle w_{1,1},\ldots,w_{1,s}\rangle]]\cong A[[x_{1,1},\ldots,x_{1,s}]]$. 
We can rewrite the polynomials $g_j(w_{1,j})$ as monic polynomials 
$f_j(x_{1,j})$ in  $x_{1,j}$ with coefficients in $A$.  By the Weierstrass
preparation theorem, one has $f_j(x_{1,j}) = h_j(x_{1,j})\cdot u_j(x_{1,j})$,
where $h_j(x)$ is a monic polynomial in $A[x]$ whose non-leading coefficients lie
in the maximal ideal of $A$ and $u_j(x)$ is a unit power series in $A[[x]]$.  
Since $x_{1,j}$ lies in every maximal ideal of $B$ and $u_j(x_{1,j})$ has invertible image 
in $B/B\cdot x_{1,j}$,
it follows that 
$u_j(x_{1,j})$ is a unit in $B$. Hence
$h_j(x_{1,j})$ annihilates $M$ for all $j$.   
Let $B_{f_1,\ldots,f_s}=A[[x_{1,1},\ldots,x_{1,s}]]/I_{h_1,\ldots,h_s}$, where
$I_{h_1,\ldots,h_s}$ is the ideal in $A[[x_{1,1},\ldots,x_{1,s}]]$ generated by
$h_{1,j}(x_{1,j})$ for $1\le j\le s$.  Then $B_{f_1,\ldots,f_s}$ is free and
finitely generated as an abstract $A$-module.  
The ring  $D=B/I$ is isomorphic to the group ring 
$B_{f_1,\ldots,f_s}[Q\times Q']$, which implies that $D$ is free and
finitely generated as an abstract $A$-module.  Since, as noted above, $I$ is a
closed ideal in $B$ and $D=B/I$, it follows that $D$ is a pseudocompact
$A$-algebra that is a pseudocompact $B$-module. Because $M$ is
a pseudocompact $D$-module that is finitely generated as an abstract $A$-module,
there is a surjective homomorphism
$\bigoplus_{i=1}^z D \to M$ of pseudocompact $D$-modules for some finite
number $z$. Since this homomorphism is also a homomorphism of pseudocompact
$B$-modules, this proves Proposition 
\ref{prop:secondstep} if $\Delta$ is as in Lemma \ref{lem:efissocute}(ii).

\medbreak

Suppose now that $\Delta$ is as in Lemma \ref{lem:efissocute}(iii) and $r=1$.
Write $w_2$ instead of $w_{2,1}$. 
Let $\tilde{\Delta}$ be the subgroup of $\Delta$ that is topologically generated
by $w_1=\overline{\Phi}^d$ and by $w_2$. Then $\tilde{\Delta}$
has finite index $d\,|\tilde{\Delta}_1|$ in $\Delta$. Let $\tilde{B}=A[[\tilde{\Delta}]]$.
Suppose we prove that there is a pseudocompact $\tilde{B}$-module
$\tilde{F}$ that is free and finitely generated as an abstract $A$-module and
a surjective homomorphism $\tilde{\varphi}:\tilde{F}\to M$ of pseudocompact
$\tilde{B}$-modules. Then the induced module
$F=\mathrm{Ind}_{\tilde{\Delta}}^\Delta (\tilde{F})$ is a pseudocompact
$B$-module that is free and finitely generated as an abstract $A$-module and
$\tilde{\varphi}$ induces a surjective homomorphism
$\varphi:F\to M$ of pseudocompact $B$-modules. 
Hence we are reduced to proving Proposition \ref{prop:secondstep} for 
$\tilde{\Delta}$.

By Proposition \ref{prop:firststep} and Lemma \ref{lem:represent}, 
there exist integers $s\ge 0$ and $N'\ge 1$ such that 
$(w_2^{p^s}-1)^{N'}\cdot M=\{0\}$.
By $(\ref{eq:commute})$, the left ideal $\tilde{J}=\tilde{B}
\cdot (w_2^{p^s}-1)^{N'}$ is a two-sided ideal in $\tilde{B}$. 
Moreover, it is closed in $\tilde{B}$ by Remark \ref{rem:useful}(i). 
Let $x_1=w_1-1$, so that $A[[x_1]]\cong A[[\langle w_1\rangle]]$, and define
$$A_{\tilde{J}}=A[[\langle w_2\rangle]]/
\left( (w_2^{p^s}-1)^{N'}\right).$$
Since  $(w_2^{p^s}-1)^{N'}$ is a monic polynomial in $(w_2-1)$
whose non-leading coefficients lie in the maximal ideal of $A$,
$A_{\tilde{J}}$ is free and finitely generated as an abstract $A$-module.
Every element in 
$\tilde{D}=\tilde{B}/\tilde{J}$ can be written in a unique way as a convergent 
power series
\begin{equation}
\label{eq:theformhere}
\sum_{i=0}^{\infty} \,  a_i\, x_1^i,
\quad \mbox{where each $a_i$ lies in $A_{\tilde{J}}$.}
\end{equation}
Moreover, any choice of $a_i\in A_{\tilde{J}}$, for all $i\ge 0$,
defines an element in $\tilde{D}$.

Using the Weierstrass preparation theorem in $A[[x_1]]\cong A[[\langle w_1\rangle]]$
and arguing similarly to the case when $\Delta$ is as in Lemma \ref{lem:efissocute}(ii),
it follows that there exists a monic polynomial
$$f_1(x) = x^n+b_{n-1}x^{n-1}+\cdots + b_0 \in A[x]$$
whose non-leading coeffcients are in the maximal ideal $m_A$ of $A$ such that 
$f_1(x_1)$ annihilates $M$.  

Let $\tilde{D}\cdot f_1(x_1)$ be the left ideal in $\tilde{D}$ that is generated
by $f_1(x_1)$. Consider the natural surjective $A$-module homomorphism
$$\beta: \quad\bigoplus_{i=0}^{n-1} A_{\tilde{J}}\;x_1^i
\longrightarrow \tilde{D}/(\tilde{D}\cdot f_1(x_1))$$
which sends $\sum_{i=0}^{n-1} \,  a_i\,x_1^i$ to the corresponding residue class of 
$\sum_{i=0}^{n-1} \,  a_i\,x_1^i$ modulo 
$\tilde{D}\cdot f_1(x_1)$.
We claim that $\beta$ is injective. Suppose there exists an element
$t=\sum_{i=0}^\infty \, a_i\, x_1^i$ in $\tilde{D}$ such that
\begin{equation}
\label{eq:step21}
t\cdot f_1(x_1)= (a_0+a_1x_1+\cdots + a_ix_1^i+\cdots) \cdot 
(x_1^n+b_{n-1}x_1^{n-1}+\cdots + b_0)
\end{equation}
lies in $\bigoplus_{i=0}^{n-1} A_{\tilde{J}} \,x_1^i$. Since the 
$a_i$ lie in $A_{\tilde{J}}$ and the $b_j$ lie in 
$m_A\subset A$, the $b_j$ commute with the $a_i$. 
Since all the $b_j$ lie in $m_A$, we see that all the $a_i$ lie in 
$m_A\cdot A_{\tilde{J}}$. Iterating this process, it follows, using induction,
that all the $a_i$ lie in $(m_A)^c\cdot A_{\tilde{J}}$ 
for all $c\ge 1$. This means that all the $a_i$ have to be zero. Thus $\beta$
is injective, which implies that
$\tilde{D}/(\tilde{D}\cdot f_1(x_1))\cong \bigoplus_{i=0}^{n-1} 
A_{\tilde{J}}\,x_1^i$
as abstract $A$-modules. Since we have already noted that 
$A_{\tilde{J}}$ is free and finitely generated as an abstract $A$-module, 
it follows that $\tilde{D}/ (\tilde{D}\cdot f_1(x_1))$ is free and finitely generated 
as an abstract $A$-module. 
By Remark \ref{rem:useful}(i), it follows that 
$\tilde{D}/(\tilde{D}\cdot f_1(x_1))$ is a pseudocompact $\tilde{D}$-module,
and hence, since $\tilde{D}=\tilde{B}/\tilde{J}$, also a pseudocompact
$\tilde{B}$-module. Because $M$ is a pseudocompact $\tilde{D}$-module that is 
finitely generated as an abstract $A$-module,
there is a surjective homomorphism
$\bigoplus_{i=1}^z \tilde{D} \to M$ of pseudocompact $\tilde{D}$-modules for some finite
number $z$. Since this homomorphism is also a homomorphism of pseudocompact
$\tilde{B}$-modules, this proves Proposition 
\ref{prop:secondstep} if $\Delta$ is as in Lemma \ref{lem:efissocute}(iii) and $r=1$.


\subsection{The case $r>1$ for $\Delta$ as in Lemma \ref{lem:efissocute}(iii)}
\label{ss:generalR} 

In this section, we complete the proof of Theorem \ref{thm:mainthm} by
considering the case when $\Delta$ is as in Lemma \ref{lem:efissocute}(iii)
and $r>1$. As before, let $B=A[[\Delta]]$.
We make the following adjustments to Propositions \ref{prop:zerostep} -
\ref{prop:secondstep}.

In Proposition \ref{prop:zerostep}, we  consider ideals of the form
$J_j=B\cdot (w_{2,j}^{N_j}-1)^{N'_j}$ for positive integers $N_j,N'_j$ for
all $1\le j\le r$ and define $J_{(j)}=J_1+\cdots + J_j$. 
As in Lemma \ref{lem:represent}, it follows that
$J_j=B \cdot (w_{2,j}^{p^{s_j}} - 1)^{N'_j}$, where $p^{s_j}$ is the maximal power of
$p$ dividing $N_j$. Using
$(\ref{eq:commute})$ and Remark \ref{rem:useful}(i), we see that $J_{(j)}$ is a 
closed two-sided ideal in $B$. Hence the quotient ring 
$\overline{B}_j=B/J_{(j)}$ is a pseudocompact $A$-algebra. 
Letting $\overline{J}_j=\overline{B}_{j-1}\cdot (w_{2,j}^{N_j}-1)^{N'_j}$ in 
$\overline{B}_{j-1}$, where we set $\overline{B}_0=B$, we similarly see that
$\overline{J}_j=\overline{B}_{j-1}\cdot(w_{2,j}^{p^{s_j}} - 1)^{N'_j}$ is a closed two-sided ideal in $\overline{B}_{j-1}$.
The last statement to be shown in generalizing Proposition \ref{prop:zerostep} is that 
$\overline{J}_j$ is a topologically free
rank one left $\overline{B}_{j-1}$-module and a topologically free rank one 
right $\overline{B}_{j-1}$-module.  

To show this last statement, we use a suitable cofinal system of closed normal finite index 
subgroups of $\Delta$ to prove the following.
Every element of $\overline{B}_{j-1}$ can be written in a 
unique way as a convergent power series
\begin{equation}
\label{eq:theform1}
\sum \;z_{u,a,\xi,b_j,c_j,\ldots,c_r} \,\sigma^u (w_1 - 1)^a \,\xi
\left(\prod_{i=j+1}^r(w_{2,i}-1)^{c_{i}}\right)
w_{2,j}^{b_j} (w_{2,j}^{p^{s_j}} - 1)^{c_j}
\end{equation}
\begin{equation}
\label{eq:theformalt1}
\left(\mathrm{resp. }\quad\sum\; 
\omega_{u,a,\xi,b_j,c_j,\ldots,c_r}\,(w_{2,j}^{p^{s_j}} - 1)^{c_j} w_{2,j}^{b_j}
\left(\prod_{i=j+1}^r(w_{2,i}-1)^{c_i}\right)\xi\, (w_1 - 1)^a \sigma^u\right)
\end{equation}
in which the sum ranges over all tuples $(u,a,\xi,b_j,c_j,\ldots,c_r)$
with $0\le u\le d-1$, $a\ge 0$, $\xi\in\tilde{\Delta}_1$, $0\le b_j\le p^{s_j} - 1$
and $c_j,\ldots,c_r \ge 0$, and each $z_{u,a,\xi,b_j,c_{j},\ldots,c_r}$ 
(resp. $\omega_{u,a,\xi,b_j,c_j,\ldots,c_r}$) lies in 
$$A_{(j-1)}=A[[\langle w_{2,1},\ldots, w_{2,j-1}\rangle]]/
\left( (w_{2,1}^{p^{s_1}}-1)^{N'_1},\ldots,(w_{2,j-1}^{p^{s_{j-1}}}-1)^{N'_{j-1}}\right).$$
Moreover, any choice of $z_{u,a,\xi,b_j,c_{j},\ldots,c_r}$ 
(resp. $\omega_{u,a,\xi,b_j,c_j,\ldots,c_r}$) in $A_{(j-1)}$ 
defines an element in $\overline{B}_{j-1}$.

In Proposition \ref{prop:firststep}, let $M$ be a
pseudocompact $B$-module that is finitely generated as an abstract $A$-module.
Using the same arguments as in the case when $r=1$, 
it follows that
for each $1\le j\le r$, there exist positive integers $N_j,N'_j$ such that 
$(w_{2,j}^{N_j}-1)^{N'_j}\cdot M=\{0\}$.  

In Proposition \ref{prop:secondprop}, we replace in part (ii) the ideal $J$ by
an ideal of the form $J=J_1+\cdots+J_r$,  where for $1\le j\le r$, 
$J_j=B \cdot(w_{2,j}^{N_j } - 1)^{N'_j}$ for certain integers $N_j,N'_j\ge 1$.
Then, as before, $J_j=B\cdot (w_{2,j}^{p^{s_j} } - 1)^{N'_j}$, where $p^{s_j}$ is the 
maximal power of $p$ dividing $N_j$, and $J$ is a closed two-sided ideal in $B$.
Suppose $M$ is a pseudocompact module for $\Lambda=B/J$ that is finitely
generated as an abstract $\Lambda$-module and
$T$ is a pseudocompact $\Lambda$-submodule of $M$
that is finitely generated as an abstract $A$-module. We need
to prove the existence of a pseudocompact $\Lambda$-submodule $M'$ of 
$M$ such that $M' \cap T = \{0\}$ and $M/M'$ is finitely generated as an 
abstract $A$-module. 

To prove this statement,
we proceed as for $r=1$ and first consider the case when $N_j'=1$ for
all $1\le j\le r$. In this case, $\Lambda=B/J=A[[\overline{\Delta}]]$, where
$\overline{\Delta}$ is the quotient of $\Delta$ by the closed normal
subgroup that is topologically generated by 
$w_{2,j}^{p^{s_j}}$ for $1\le j\le r$.  Using similar
arguments as in the case when $r=1$, we find a
pseudocompact $\Lambda$-submodule $M'$ of $M$ having the desired properties
if $N_j'=1$ for all $1\le j\le r$. 
For arbitrary $N_j'$, we replace $M(\epsilon^m)$ in $(\ref{eq:mmm})$ by
\begin{equation}
\label{eq:mmm1}
M(\epsilon_1^{m_1},\ldots,\epsilon_r^{m_r}) = \{\alpha \in M\;|\; \epsilon_j^{m_j}
 \cdot \alpha = 0\quad\mbox{for $1\le j\le r$}\}
 \end{equation}
for $m_1,\ldots,m_r \ge 1$, and prove analogous statements to the ones
in Lemma \ref{lem:fisit}, Corollary \ref{cor:neoth} and Lemma \ref{lem:itsfinite} 
to find $M'$.

Proposition \ref{prop:bounded} stays the same as before.
To prove Proposition \ref{prop:secondstep} for $r>1$, we let 
$\tilde{\Delta}$ be the subgroup of $\Delta$ that is topologically generated
by $w_1$ and by $w_{2,j}$ for $1\le j\le r$. Then $\tilde{\Delta}$
has finite index $d\,|\tilde{\Delta}_1|$ in $\Delta$. One argues
as in the case when $r=1$, that it is enough to prove Proposition \ref{prop:secondstep}
for $\tilde{B}=A[[\tilde{\Delta}]]$. 
By Proposition \ref{prop:firststep} and Lemma \ref{lem:represent}, 
there exist integers $s_j\ge 0$ and $N'_j\ge 1$ such that 
$(w_{2,j}^{p^{s_j}}-1)^{N'_j}\cdot M=\{0\}$ for $1\le j\le r$.
Using $(\ref{eq:commute})$ and Remark \ref{rem:useful}(i), it follows
that $\tilde{J}=\tilde{J}_1+\cdots +\tilde{J}_r$ is a closed two-sided ideal in $\tilde{B}$,
where $\tilde{J}_j=\tilde{B}\cdot (w_{2,j}^{p^{s_j}}-1)^{N'_j}$.
Define
$$A_{\tilde{J}}=A[[\langle w_{2,1},\ldots, w_{2,r}\rangle]]/
\left( (w_{2,1}^{p^{s_1}}-1)^{N'_1},\ldots,(w_{2,r}^{p^{s_r}}-1)^{N'_r}\right).$$
Then $A_{\tilde{J}}$ is free and finitely generated as an abstract $A$-module,
and every element in $\tilde{D}=\tilde{B}/\tilde{J}$ can be written in a 
unique way as a convergent power series as in $(\ref{eq:theformhere})$.
We can now proceed using the same arguments as in the case when $r=1$
to complete the proof  for the case when $r>1$.

\medbreak

The proof of Theorem \ref{thm:mainthm} in the case when $G$ is replaced by 
$\Delta$ as in Lemma \ref{lem:efissocute}(iii) and $r>1$ follows the same three 
steps as in the case when $r=1$. 

In the proof of step $1$, we need to use an inductive argument as follows.
As in the case when $r=1$, suppose $P^\bullet$ has properties (i), (ii) and (iii) 
of Theorem \ref{thm:derivedresult} and suppose that $n_2=0$,
so that $P^i = 0$ if $i > 0$.  
Since $\HH^i(P^\bullet) = 0$ if $i < n_1$, $P^\bullet$ is isomorphic in $D^-(B)$ to
the complex $P_0^\bullet$ which is obtained from $P^\bullet$ by replacing $P^{n_1}$ 
by $P^{n_1}/\BB^{n_1}(P^\bullet)$ and $P^i$ by $0$ for $i < n_1$. 
Define $J_{(0)}=\{0\}$ and $\overline{B}_0=B/J_{(0)}$. Assume by
induction that for $1\le j\le r$, $P^\bullet$ is isomorphic in $D^-(B)$ to a complex
$P_{j-1}^\bullet$ such that $P_{j-1}^i = 0$ if $i > 0$ or $i<n_1$ and such that
if $n_1\le i\le 0$ then $P_{j-1}^i$ is annihilated by a closed two-sided
ideal $J_{(j-1)}=J_1+\cdots +J_{j-1}$, where for $1\le t\le j-1$, 
$J_t=B \cdot(w_{2,t}^{N_t } - 1)^{N'_t}$ for certain integers $N_t,N'_t\ge 1$. 
Let $\overline{B}_{j-1}=B/J_{(j-1)}$ and
view $P_{j-1}^\bullet$ as a complex in $D^-(\overline{B}_{j-1})$.
Using the above adjustments of  Propositions \ref{prop:zerostep} - \ref{prop:bounded}
and Remark \ref{rem:hartshorne}, we find a complex $P_j^\bullet$
which is isomorphic to $P_{j-1}^\bullet$ in $D^-(\overline{B}_{j-1})$ 
such that $P_j^i=0$ if $i>0$ or $i<n_1$ and such that if $0\le i\le n_1$
then $P_j^i$ is annihilated by a closed two-sided ideal 
$\overline{J}_j=\overline{B}_{j-1}\cdot(w_{2,j}^{N_j} - 1)^{N'_j}$ for certain integers
$N_j,N_j'\ge 1$.
Note that if $J_j=B\cdot (w_{2,j}^{N_j} - 1)^{N'_j}$ and $J_{(j)}=J_{(j-1)}+J_j$, then 
$\overline{B}_j=B/J_{(j)}=\overline{B}_{j-1}/\overline{J}_j$ as pseudocompact rings. 
Since $P_j^\bullet$ can be viewed as a complex in $D^-(B)$ by inflation, it follows that
$P_j^\bullet$ is isomorphic to $P_{j-1}^\bullet$, and thus to $P^\bullet$, in $D^-(B)$.
Hence step $1$ follows by induction.

Steps $2$ and $3$ of the proof of Theorem \ref{thm:mainthm} are proved in the same 
way as when $r=1$, using the above adjustments of Propositions \ref{prop:secondprop} 
and \ref{prop:secondstep}.


\section{An example}
\label{s:ex}
\setcounter{equation}{0}

In this section, we want to revisit an example that was considered in \cite{bcderived} concerning the
deformations of group cohomology elements. Let $\ell >2$ be a rational prime with $\ell\equiv 3\mod 4$ and
let $G=\mathrm{Gal}(\overline{\mathbb{Q}}_\ell/\mathbb{Q}_\ell)$. Let $k=\mathbb{Z}/2$ and
$W=\mathbb{Z}_2$, and let $M=k$ have trivial $G$-action. Because of the Kummer sequence
$$1\to\{\pm 1\}\to \overline{\mathbb{Q}}_\ell^*\xrightarrow{\cdot 2}\overline{\mathbb{Q}}_\ell^*\to 1$$
we obtain that $\HH^2(G,M)=\mathbb{Z}/2$ has exactly one non-trivial element $\beta$. Moreover,
it was shown in \cite{bcderived} that the mapping cone $C(\beta)^\bullet$ is isomorphic to
$V^\bullet[1]$ for a two-term complex $V^\bullet$ that is concentrated in 
degrees $-1$ and $0$
\begin{equation}
\label{eq:v}
V^\bullet:\qquad \cdots 0 \to k[G_b]\xrightarrow{d} k[G_a]\to 0\cdots,
\end{equation}
where $a=\ell$, $b$ is an element of $\mathbb{Z}_\ell^*$ that is not a square mod $\ell$,
$G_a=\mathrm{Gal}(\mathbb{Q}_\ell(\sqrt{a})/\mathbb{Q}_\ell)$, 
$G_b=\mathrm{Gal}(\mathbb{Q}_\ell(\sqrt{b})/\mathbb{Q}_\ell)$ and
$d$ is the augmentation map of $k[G_b]$ composed with 
multiplication by $1+\sigma_a$ when $G_a=\{1,\sigma_a\}$.
It was also shown in \cite{bcderived} that the tangent space 
$\mathrm{Ext}^1_{D^-(k[[G]])}(V^\bullet,V^\bullet)$ is 4-dimensional over $k$, and that
the versal proflat deformation ring of $V^\bullet$ is universal and isomorphic to
$R^{\mathrm{fl}}(G,V^\bullet)\cong W[[G^{\mathrm{ab},2}]]\hat{\otimes}_W
W[[G^{\mathrm{ab},2}]]$,
where $G^{\mathrm{ab},2}$ denotes the abelianized $2$-completion of
$G$. Note that the universal proflat deformation ring $R^{\mathrm{fl}}(G,V^\bullet)$ is universal
with respect to isomorphism classes of quasi-lifts of $V^\bullet$ over objects $R$ in 
$\hat{\mathcal{C}}$ whose cohomology groups are topologically flat, and hence
topologically free, pseudocompact $R$-modules.

We now turn to the situation when $G$ is replaced by its maximal abelian quotient 
$G^{\mathrm{ab}}$. We want to compute the versal deformation rings of several complexes
related to the above $V^\bullet$. The complexes we will consider are all inflated from 
the maximal pro-$2$ quotient $G^{\mathrm{ab},2}$ of $G^{\mathrm{ab}}$. 
By Proposition \ref{prop:prop}, it will suffice to determine their versal deformation rings
as complexes for  $\Gamma=G^{\mathrm{ab},2}$.

Since  $\ell\equiv 3\mod 4$, local class field theory shows that there are topological generators $w_1$ and
$w_2$ for $\Gamma=\mathrm{Gal}(\mathbb{Q}_\ell^{\mathrm{ab},2}/\mathbb{Q}_\ell)$ 
with the 
following properties. The element $w_2$ has order $2$ and $\{\mathrm{id},w_2\}=
\mathrm{Gal}(\mathbb{Q}_\ell^{\mathrm{ab},2}/\mathbb{Q}_\ell^{\mathrm{un},2})$ where 
$\mathbb{Q}_\ell^{\mathrm{un},2}$ is the maximal unramified pro-$2$ extension of $\mathbb{Q}_\ell$.
The element $w_1$ is a topological generator of $\mathrm{Gal}(\mathbb{Q}_\ell^{\mathrm{ab},2}/
\mathbb{Q}_\ell(\sqrt{\ell}))\cong \mathbb{Z}_2$, and $\Gamma=\langle w_1,w_2\rangle$ is
isomorphic to $\mathbb{Z}_2\times \mathbb{Z}/2$. 
Note that $w_1$ (resp. $w_2$, resp. $w_1w_2$) acts trivially on the quadratic extension
$\mathbb{Q}_\ell(\sqrt{\ell})$ (resp. $\mathbb{Q}_\ell(\sqrt{-1})$, resp. $\mathrm{Q}_\ell(\sqrt{-\ell})$).

As before, let $M=k=\mathbb{Z}/2$ have trivial $\Gamma$-action.
Since $\langle w_1\rangle=\mathbb{Z}_2$ has cohomological dimension $1$, the spectral 
sequence
$$\HH^p(\langle w_1\rangle,\HH^q(\langle w_2\rangle,M))\Longrightarrow\HH^{p+q}(\Gamma,M)$$
degenerates and we get a short exact sequence for all $s\ge 1$
$$0\to \HH^1(\langle w_1\rangle,\HH^{s-1}(\langle w_2\rangle,M))
\to \HH^s(\Gamma,M)\to \HH^s(\langle w_2\rangle,M)^{\langle w_1\rangle}\to 0.$$
Since $\HH^s(\langle w_2\rangle,M)=k$ for all $s\ge 0$ and $\HH^1(\langle w_1\rangle,k)=k$,
we obtain that
$$\HH^0(\Gamma,M)=k\quad \mbox{and}\quad \HH^s(\Gamma,M)=k\oplus k\; \mbox{for $s\ge 1$.}$$
This means that there are three non-trivial elements in $\HH^2(\Gamma,M)$. 
Let $x\in\{\ell,-1,-\ell\}$ and consider the element
$h_x$ in $\HH^1(\Gamma,\{\pm 1\})=\mathrm{Hom}(\Gamma,\{\pm 1\})$ which
corresponds to the augmentation sequence
\begin{equation}
\label{eq:gx}
0\to  k \to k[G_x]\to k \to 0
\end{equation}
where $G_x=\mathrm{Gal}(\mathbb{Q}_\ell(\sqrt{x})/\mathbb{Q}_\ell)$.
Inflating the cup product $h_a\cup h_b$ for $a,b\in\{\ell,-1,-\ell\}$ to an element in
$\HH^2(G,\{\pm 1\})$, it follows that $h_a\cup h_b$
corresponds to the Hilbert symbol $(a,b)\in \HH^2(G,\{\pm 1\})$. Hence 
$h_\ell\cup h_\ell$ and $h_\ell\cup h_{-1}$ define non-trivial elements in $\HH^2(G,\{\pm 1\})$,
whereas $h_\ell\cup h_{-\ell}$ defines a trivial element in $\HH^2(G,\{\pm 1\})$. Since
the restriction of $h_\ell\cup h_\ell$ to $\langle w_2\rangle$ is non-trivial, whereas the restriction 
of $h_\ell\cup h_{-1}$ to $\langle w_2\rangle$ is trivial, 
$h_\ell\cup h_\ell \neq h_\ell\cup h_{-1}$ in $\HH^2(\Gamma,k)$. 
It follows that
$h_\ell\cup h_\ell$, $h_\ell\cup h_{-1}$ and $h_\ell\cup h_{-\ell}$
are representatives of the three non-trivial elements in  $\HH^2(\Gamma,M)$.
We obtain three non-split two-term complexes $V_{y}^\bullet$ in $D^-(k[[\Gamma]])$
that are concentrated in degrees $-1$ and $0$
\begin{equation}
\label{eq:v3}
V_{y}^\bullet:\qquad \cdots 0 \to k[G_y]\xrightarrow{d} k[G_\ell]\to 0\cdots
\end{equation}
where $y\in\{\ell,-1,-\ell\}$ and $d$ is the augmentation map followed by 
multiplication with the trace element of $G_\ell$. In particular, for $y\in\{\ell,-1\}$, the inflation
of $V_{y}^\bullet$ to $G$ is isomorphic to $V^\bullet$.

\begin{lemma}
\label{lem:Ext1Vy}
For $y\in \{\ell,-\ell\}$ $($resp. $y=-1$$)$, the $k$-dimension of $\mathrm{Ext}^1_{D^-(k[[\Gamma]])}(
V_{y}^\bullet,V_{y}^\bullet)$ is at least $3$ $($resp. at least $4$$)$.
Moreover, the proflat tangent space $t_{F^{\mathrm{fl}}}$ is isomorphic to the tangent space
$t_F$.
\end{lemma}

\begin{proof}
Let $y\in\{\ell,-1,-\ell\}$ and consider the triangle in $D^-(k[[\Gamma]])$
\begin{equation}
\label{eq:cutetriangle}
k^\bullet[1] \xrightarrow{\gamma_y} V_{y}^\bullet \xrightarrow{\alpha_y}  
k^\bullet \xrightarrow{\beta_y} k^\bullet[2]
\end{equation}
where $k^\bullet$ stands for the one-term complex with $k$ concentrated in degree $0$
and $\beta_y=h_\ell\cup h_y$ is the non-zero
element in $\HH^2(\Gamma,k)$ associated to $V_{y}^\bullet$.

The morphism $\mathrm{Ext}^1_{D^-(k[[\Gamma]])}(k^\bullet[1],k^\bullet)
\xrightarrow{\circ\beta_y[-1]}\mathrm{Ext}^1_{D^-(k[[\Gamma]])}(k^\bullet[-1],k^\bullet)$
is injective, since it sends the identity in 
$\mathrm{Ext}^1_{D^-(k[[\Gamma]])}(k^\bullet[1],k^\bullet)=
\mathrm{Hom}_{D^-(k[[\Gamma]])}(k^\bullet[1],k^\bullet[1])\cong k$ 
to $\beta_y[-1]$ where $\beta_y$ is as in $(\ref{eq:cutetriangle})$.
Hence it follows from \cite[Prop. 9.6]{bcderived} that $t_{F^{\mathrm{fl}}}\cong t_F$.

Using long exact Hom sequences in $D^-(k[[\Gamma]])$ associated to the triangle
$(\ref{eq:cutetriangle})$, we
obtain the following diagram with exact rows and columns, where $\mathrm{Hom}$
stands for $\mathrm{Hom}_{D^-(k[[\Gamma]])}$ and $\mathrm{Ext}$ stands for
$\mathrm{Ext}_{D^-(k[[\Gamma]])}$.
\begin{equation}
\label{eq:betaoy}
\def\objectstyle{\scriptstyle}
\def\labelstyle{\scriptstyle}
\xymatrix @-1.1pc @R1.2pc{
&\mathrm{Ext}^{-1}(k^\bullet[1],k^\bullet)\ar[d]&\mathrm{Hom}(k^\bullet[1],k^\bullet[1])\ar[d]&
&\mathrm{Hom}(k^\bullet[1],k^\bullet)\ar[d]&\mathrm{Ext}^1(k^\bullet[1],k^\bullet[1])\ar[d]\\
&\mathrm{Hom}(k^\bullet,k^\bullet)\ar[d]&\mathrm{Ext}^1(k^\bullet,k^\bullet[1])\ar[d]&
&\mathrm{Ext}^1(k^\bullet,k^\bullet)\ar[r]\ar[d]&\mathrm{Ext}^2(k^\bullet,k^\bullet[1])\ar[d]\\
\mathrm{Hom}(V_{y}^\bullet,V_{y}^\bullet)\ar[r]&
\mathrm{Hom}(V_{y}^\bullet,k^\bullet)\ar[r]\ar[d]&\mathrm{Ext}^1(V_{y}^\bullet,k^\bullet[1])\ar[r]\ar[d]&
\mathrm{Ext}^1(V_{y}^\bullet,V_{y}^\bullet)\ar[r]&\mathrm{Ext}^1(V_{y}^\bullet,k^\bullet)\ar[r]\ar[d]&
\mathrm{Ext}^2(V_{y}^\bullet,k^\bullet[1])\\
&\mathrm{Hom}(k^\bullet[1],k^\bullet)&\mathrm{Ext}^1(k^\bullet[1],k^\bullet[1])\ar[d]&
&\mathrm{Ext}^1(k^\bullet[1],k^\bullet)\ar[d]&\\
&&\mathrm{Ext}^2(k^\bullet,k^\bullet[1])&&
\mathrm{Ext^2}(k^\bullet,k^\bullet)&
}
\end{equation}
Because $\mathrm{Ext}^{-1}(k^\bullet[1],k^\bullet)=0=
\mathrm{Hom}(k^\bullet[1],k^\bullet)$ in the second column of $(\ref{eq:betaoy})$,
it follows that $\mathrm{Hom}(V_{y}^\bullet,k^\bullet)\cong k$ in the third row. We conclude
that the horizontal morphism in the third row
\begin{equation}
\label{eq:ext1first}
\mathrm{Ext}^1(V_{y}^\bullet,k^\bullet[1])\to \mathrm{Ext}^1(V_{y}^\bullet,V_{y}^\bullet)
\end{equation}
is injective. 

Since the vertical morphism $\mathrm{Ext}^1(k^\bullet[1],k^\bullet[1])\to 
\mathrm{Ext}^2(k^\bullet,k^\bullet[1])$
in the sixth column and the horizontal morphism $\mathrm{Ext}^1(k^\bullet,k^\bullet)\to 
\mathrm{Ext}^2(k^\bullet,k^\bullet[1])$ in the second row of $(\ref{eq:betaoy})$ can both be
identified with the morphism $\HH^1(\Gamma,k)\to \HH^3(\Gamma,k)$ that
sends $h_x$ to $h_x\cup \beta_y$, it follows that  the composition of morphisms
$\mathrm{Ext}^1(k^\bullet,k^\bullet)\to \mathrm{Ext}^2(k^\bullet,k^\bullet[1])
\to \mathrm{Ext}^2(V_{y}^\bullet,k^\bullet[1])$
in the second row and sixth column of $(\ref{eq:betaoy})$
is the zero morphism. We conclude that the horizontal morphism in the third row
\begin{equation}
\label{eq:ext1second}
\mathrm{Ext}^1(V_{y}^\bullet,V_{y}^\bullet)\to \mathrm{Ext}^1(V_{y}^\bullet,k^\bullet)
\end{equation}
is surjective. 

Because the vertical morphism in the third column of $(\ref{eq:betaoy})$
$$k\cong \mathrm{Hom}(k^\bullet[1],k^\bullet[1])\to \mathrm{Ext}^1(k^\bullet,k^\bullet[1])
\cong \HH^1(\Gamma,k)\cong k\oplus k$$ 
has cokernel of $k$-dimension at least $1$, it follows that
$\mathrm{dim}_k\,\mathrm{Ext}^1(V_{y}^\bullet,k^\bullet[1])$ is at least $1$.
The vertical morphism 
$\mathrm{Ext}^1(k^\bullet[1],k^\bullet[1])\to \mathrm{Ext}^2(k^\bullet,k^\bullet[1])$
in the third column of $(\ref{eq:betaoy})$
can be identified with the morphism 
$\HH^1(\Gamma,k)\to \HH^3(\Gamma,k)$ that sends 
$h_x$ to $h_x\cup \beta_y$. Since $h_{-1}\cup h_{-1}$
is inflated from an element in $\HH^2(\langle w_1\rangle, k)$ and 
$\langle w_1\rangle=\mathbb{Z}_2$
has cohomological dimension $1$, it follows that for $y=-1$, the vertical morphism in the third column
$\mathrm{Ext}^1(k^\bullet[1],k^\bullet[1])\to \mathrm{Ext}^2(k^\bullet,k^\bullet[1])$
has non-trivial kernel. Hence $\mathrm{dim}_k\,\mathrm{Ext}^1(V_{y}^\bullet,k^\bullet[1])$ 
is at least $2$ if $y=-1$. 

Since $\mathrm{Hom}(k^\bullet[1],k^\bullet)=0$ and the vertical morphism in the fifth column 
$\mathrm{Ext}^1(k^\bullet[1],k^\bullet)\to\mathrm{Ext}^2(k^\bullet,k^\bullet)$ is injective, it follows
that the vertical morphism 
$\mathrm{Ext}^1(k^\bullet,k^\bullet)\to \mathrm{Ext}^1(V_{y}^\bullet,k^\bullet)$ is an isomorphism.
Because $\mathrm{Ext}^1(k^\bullet,k^\bullet)\cong\HH^1(\Gamma,k)\cong k\oplus k$, this implies
that $\mathrm{Ext}^1(V_{y}^\bullet,k^\bullet)$ has $k$-dimension $2$. 
Using $(\ref{eq:ext1first})$ and $(\ref{eq:ext1second})$, this implies Lemma \ref{lem:Ext1Vy}.
\end{proof}

\begin{thm}
\label{thm:GabBigOne}
For $y\in \{\ell,-1,-\ell\}$, the versal deformation ring $R(\Gamma,V_{y}^\bullet)$ and the
versal proflat deformation ring $R^{\mathrm{fl}}(\Gamma,V_{y}^\bullet)$ have the following
isomorphism types:
\begin{eqnarray*}
&R(\Gamma,V_{\ell}^\bullet)\cong W[[t_1,t_2,t_3]]/(t_2t_3(2+t_3))\quad \mbox{and} 
\quad R^{\mathrm{fl}}(\Gamma,V_{\ell}^\bullet)\cong W[[t_1,t_2,t_3]]/(t_3(2+t_3)),&
\\
&R(\Gamma,V_{-\ell}^\bullet)\cong
R^{\mathrm{fl}}(\Gamma,V_{-\ell}^\bullet)\cong W[[t_1,t_2,t_3]]/(t_3(2+t_3)),&
\\
&R(\Gamma,V_{-1}^\bullet)\cong
R^{\mathrm{fl}}(\Gamma,V_{-1}^\bullet)\cong W[[t_1,t_2,t_3,t_4]]/(t_2(2+t_2),
t_4(2+2t_2-t_3t_4)).&
\end{eqnarray*}
\end{thm}

\begin{proof}
We first give an outline of the proof. We will show that every quasi-lift of $V_{y}^\bullet$ over a ring $A$
in $\hat{\mathcal{C}}$ can be represented by a two-term complex $P^\bullet:P^{-1}
\xrightarrow{d_P} P^0$, concentrated in degrees $-1$ and $0$, in which $P^{-1}$ and $P^0$
are pseudocompact $A[[\Gamma]]$-modules that are free of rank two over $A$. We will
show further that the action of the topological generators $w_1$ and $w_2$ of $\Gamma$ on
$P^{-1}$ and $P^0$, as well as the differential $d_P$, are described by $2\times 2$ matrices
over $A$ whose entries satisfy certain equations. We construct a candidate for the versal deformation
ring $R(\Gamma,V_{y}^\bullet)$ by taking the completion of the ring obtained by adjoining to $W$
indeterminates corresponding to these matrix entries which are required to satisfy the above 
equations. We prove that $R(\Gamma,V_{y}^\bullet)$ is the versal deformation ring by showing
that each $P^\bullet$ as above is a specialization of the resulting quasi-lift 
$(U(\Gamma,V_{y}^\bullet),\phi_U)$ over $R(\Gamma,V_{y}^\bullet)$ and by showing that the tangent
space of $R(\Gamma,V_{y}^\bullet)$ has the correct dimension. The last step uses Lemma
$\ref{lem:Ext1Vy}$.

Let $y\in\{\ell,-1,-\ell\}$. Let $A$ be in $\hat{\mathcal{C}}$ and let
$(L^\bullet,\phi_L)$ be a quasi-lift of $V_{y}^\bullet$ over $A$. 
Since $\HH^0(V_{y}^\bullet)=k$, it follows that  $\HH^0(L^\bullet)$ is a quotient  of $A$.
By Theorem \ref{thm:derivedresult} and Remark \ref{rem:dumbdumb},
we can thus assume that $L^\bullet$ is a two-term complex,
concentrated in degrees $-1$ and $0$,
$$L^\bullet:\qquad  \cdots 0 \to L^{-1}\xrightarrow{d_L} A[[\Gamma]]\to 0 \cdots$$
where $L^{-1}$ is topologically free over $A$, and such that we have an exact sequence
in $C^-(A[[\Gamma]])$
\begin{equation}
\label{eq:bb1}
0\to \HH^{-1}(L^\bullet) \to L^{-1}\xrightarrow{d_L} A[[\Gamma]]\to \HH^0(L^\bullet)\to 0.
\end{equation}
Since $\HH^0(L^\bullet)$ is a quotient  of $A$,
$w_1$ (resp. $w_2$) acts on $\HH^0(L^\bullet)$ as a scalar $s_1$
(resp. $s_2$) in $A^*$.

Because $w_1$ acts on $\HH^0(L^\bullet)$ as the scalar $s_1$, $(w_1-s_1)$
annihilates $\HH^0(L^\bullet)$. Since $A[[\Gamma]](w_1-s_1)$ is a free rank one 
pseudocompact $A[[\Gamma]]$-module that is a submodule of $A[[\Gamma]]$, 
there exists a free rank one pseudocompact
$A[[\Gamma]]$-module $F$ that is a submodule of $L^{-1}$ such that $d_L$ maps
$F$ isomorphically onto $A[[\Gamma]](w_1-s_1)$. Hence the exact sequence $(\ref{eq:bb1})$
leads to an exact sequence
in $C^-(A[[\Gamma]])$
\begin{equation}
\label{eq:bb2}
\xymatrix @R.7pc{
0\ar[r]& \HH^{-1}(L^\bullet)\ar[r]& P^{-1}\ar[r]^{d_P}& P^0\ar[r]^\mu&  \HH^0(L^\bullet)\ar[r]& 0\\
&&L^{-1}/F \ar@{=}[u]&A[[\Gamma]]/A[[\Gamma]](w_1-s_1)\ar@{=}[u]}
\end{equation}
where $P^0\cong A\langle w_2\rangle$ and $w_1$ acts on $P^0$ as multiplication by $s_1$.
Since $w_2$ has order $2$, it follows that $\{1,w_2\}$ is an $A$-basis
of $P^0$. With respect to this $A$-basis, $w_1$ (resp. $w_2$) acts on $P^0$ as the matrix
\begin{equation}
\label{eq:p0actions}
W_{1,P^0}=\left(\begin{array}{cc}s_1&0\\0&s_1\end{array}\right)
\quad \left(\mbox{resp. }W_{2,P^0}=\left(\begin{array}{cc}0&1\\1&0\end{array}\right)
\right).
\end{equation}
Moreover, $L^\bullet$ is quasi-isomorphic to the two-term complex
\begin{equation}
\label{eq:pp}
P^\bullet:\qquad \cdots 0 \to P^{-1} \xrightarrow{d_P} P^0\to 0\cdots
\end{equation}
concentrated in degrees $-1$ and $0$, where $P^{-1}$, $P^0$ and $d_P$ are as in $(\ref{eq:bb2})$.
Let $\phi_P:k\hat{\otimes}_AP^\bullet\to V_y^\bullet$ be an isomorphism in $D^-(k[[\Gamma]])$ such that
$(L^\bullet,\phi_L)$ and $(P^\bullet,\phi_P)$ are isomorphic quasi-lifts of $V^\bullet_y$ over $A$.
Since $L^\bullet$, and hence $P^\bullet$, has finite pseudocompact $A$-tor dimension at $-1$,
$P^{-1}$ is topologically flat, and hence topologically free, over $A$.
Because $k\hat{\otimes}_A P^\bullet$ must be isomorphic to $V_{y}^\bullet$ in $D^-(k[[\Gamma]])$,
it follows that $k\hat{\otimes}_AP^{-1}$ has $k$-dimension $2$, and hence
$P^{-1}$ is free over $A$ of rank $2$.

Let $K^0$ be the kernel of the morphism $\mu:P^0\to  \HH^0(L^\bullet)$ in $(\ref{eq:bb2})$.
Because $w_2$ acts on $\HH^0(L^\bullet)$ as the scalar $s_2\in A^*$,
$K^0$ contains the element $-s_2\cdot 1 + 1\cdot w_2=(-s_2,1)$ which
generates a free rank one $A$-submodule of $P^0$. Since $K^0$ is an $A[[\Gamma]]$-submodule
of $P^0$, we also have that 
$$w_2\cdot (-s_2,1)=(1,-s_2)=-s_2(-s_2,1) + (1-s_2^2,0)$$
is an element of $K^0$, and thus $(1-s_2^2,0)\in K^0$. On the other hand, $k\hat{\otimes}_A
K^0$ has $k$-dimension at most $2$, since $K^0$ is a homomorphic image of $P^{-1}$. Hence
$K^0$ is generated
by one or two elements, depending on whether $\HH^0(L^\bullet)$ is flat over $A$ or not.
If $(c,d)$ is an arbitrary element in $K^0$, then
$$(c,d) = d\cdot (-s_2,1) + (c+ds_2,0),$$
and hence $K^0$ is generated by $(-s_2,1)$ and an element of the form
$(\lambda,0)$ such that $\lambda$ divides $(1-s_2^2)$. It follows that 
$\HH^0(L^\bullet)\cong A/(\lambda)$ and $\HH^0(L^\bullet)\neq\{0\}$. 
In particular, $\HH^0(L^\bullet)$ is flat over $A$ if and only if $\lambda=0$, in which case we must
have $1-s_2^2=0$. Since the image of $d_P$ in $(\ref{eq:bb2})$ 
must be equal to $K^0$ and since $(-s_2,1)$ generates a free
$A$-module of rank $1$, we can lift $(-s_2,1)$ to a basis element $z_1$ of $P^{-1}$. 
If $\lambda=0$, then
$\HH^{-1}(L^\bullet)=\mathrm{Ker}(d_P)\cong A$ as $A$-modules, and
we choose $z_2$ to be a basis element of $P^{-1}$ which lies in $\mathrm{Ker}(d_P)$.
If $\lambda\neq 0$, then $K^0$ is generated as $A$-module by $(-s_2,1)$ and $(\lambda,0)$, and
we let $z_2$ be a preimage under $d_P$ of $(\lambda,0)$. 
Since $(\lambda,0)$ is not an $A$-multiple of $(-s_2,1)$ if $\lambda\neq 0$,
the homomorphism $k\hat{\otimes}_AP^{-1}\to
k\hat{\otimes}_AK^0$ induced by the surjection $P^{-1}\xrightarrow{d_P}K^0$ must 
be an isomorphism of two-dimensional $k$-vector spaces. 
It follows that $\{z_1,z_2\}$ is an $A$-basis of $P^{-1}$ for all $\lambda$.

With respect to the $A$-basis $\{z_1,z_2\}$ of $P^{-1}$ and
the $A$-basis $\{1,w_2\}$ of $P^0$, $d_P:P^{-1}\to P^0$ is given by
the matrix 
\begin{equation}
\label{eq:dp}
D_P=\left(\begin{array}{cc}-s_2&\lambda\\1&0\end{array}\right)
\end{equation}
when we write basis vectors as column vectors. In particular, $\HH^{-1}(L^\bullet)=\mathrm{Ker}(d_P)
\cong \mathrm{Ann}_A(\lambda)$, which implies that $(P^\bullet,\phi_P)$ is
a proflat quasi-lift of $V_{y}^\bullet$ over $A$ if and only if $\lambda=0$. 
Additionally, if $\lambda=0$ then $1-s_2^2=0$.

Suppose now that $\lambda\neq 0$. In particular, $\lambda$ is not a unit since $\HH^0(L^\bullet)
\cong A/(\lambda)$ is not $0$. Then $1-s_2^2=\lambda t_2$ for some $t_2\in A$, and
$K^0$ is generated as $A$-module by 
$(-s_2,1)$ and $(\lambda,0)$. The action of $w_1$ on $K^0$ is given by the scalar
matrix $s_1$. 
The action of $w_2$ on $P^0$ sends $(-s_2,1)$ to
$$(1,-s_2)=-s_2(-s_2,1)+(1-s_2^2,0)=-s_2(-s_2,1)+t_2(\lambda,0)$$
and $(\lambda,0)$ to
$$(0,\lambda)=\lambda(-s_2,1)+s_2(\lambda,0).$$
To obtain the action of $w_2$ on $P^{-1}$, we use the matrix representation $D_P$ of $d_P$
from $(\ref{eq:dp})$ with respect to the $A$-basis $\{z_1,z_2\}$ of $P^{-1}$. 
This means that the kernel of $d_P$ is given by $\mathrm{Ann}_A(\lambda)\cdot z_2$. Hence
the action of $w_1$ (resp. $w_2$) on $P^{-1}$ has the form
\begin{equation}
\label{eq:actionsVll}
\left(\begin{array}{cc}s_1&0\\x_1&s_1+y_1\end{array}\right)
\quad \left(\mbox{resp. }
\left(\begin{array}{cc}-s_2&\lambda \\t_2+x_2&s_2+y_2\end{array}\right)\right)
\end{equation}
for certain elements $x_1,y_1,x_2,y_2\in \mathrm{Ann}_A(\lambda)$.
If $t_2$ were not a unit, then  $w_1$ and $w_2$ would both act  trivially on
$k\hat{\otimes}_AP^{-1}$. Hence $k\hat{\otimes}_AP^\bullet$ would correspond to the
cup product of $h_\ell$ with the trivial character $h_1$ which defines the trivial element
in $\HH^2(\Gamma,\{\pm 1\})$. This is a contradiction to $k\hat{\otimes}_AP^\bullet
\cong V_{y}^\bullet$ in $D^-(k[[\Gamma]])$. Thus $t_2$ must be a unit, which 
implies $(\lambda)=(1-s_2^2)$.
But then the action of $w_1$ (resp. $w_2$) on $k\hat{\otimes}_AP^{-1}$ is
trivial (resp. non-trivial). Hence $k\hat{\otimes}_AP^\bullet$ corresponds to the
cup product $h_\ell\cup h_\ell$. We conclude that if  $\lambda\neq 0$ then 
$y=\ell$.

\medskip

We now concentrate on the case when $y=\ell$. The cases when $y=-\ell$ or $y=-1$ are
treated similarly.

Given an arbitrary quasi-lift of $V_{\ell}^\bullet$ over a ring $A$ in $\hat{\mathcal{C}}$,
we can assume this quasi-lift is given by $(P^\bullet,\phi_P)$ with $P^\bullet$ as in $(\ref{eq:pp})$.
The complex $k\hat{\otimes}_AP^\bullet$ defines $h_\ell\cup h_\ell\in \HH^2(\Gamma,k)$,
and it follows from the construction of $P^\bullet$ in $(\ref{eq:pp})$
that $h_\ell\cup h_\ell=h_\ell\cup h'$, where $h'\in \HH^1(\Gamma,k)$ 
is the class of $0\to k \to k\hat{\otimes}_AP^{-1}\to k\to 0$. Hence $h_\ell\cup h_\ell=h_\ell\cup h'$ implies
$h'=h_\ell$. So $w_1$ (resp. $w_2$)  acts trivially (resp. non-trivially) on 
$k\hat{\otimes}_AP^{-1}$. 

Since  $w_2$ acts non-trivially on $k\hat{\otimes}_AP^{-1}$, it follows that
$1\hat{\otimes}z_1$ and $w_2\cdot (1\hat{\otimes}z_1)=1\hat{\otimes}(w_2\cdot z_1)$ 
form a $k$-basis of $k\hat{\otimes}_AP^{-1}$. Since $A$ is
a commutative local ring, this implies that $\{z_1,w_2\cdot z_1\}$ is an $A$-basis of $P^{-1}$. 
It follows that with respect to the $A$-basis $\{z_1,w_2\cdot z_1\}$ of $P^{-1}$ and
the $A$-basis $\{1,w_2\}$ of $P^0$, $d_P:P^{-1}\to P^0$ is given by the matrix 
\begin{equation}
\label{eq:newdp}
\tilde{D}_P=\left(\begin{array}{cc}-s_2&1\\1&-s_2\end{array}\right).
\end{equation}
Considering the actions of $w_1$ (resp. $w_2$) on $P^{-1}$ and $P^0$, we obtain that
with respect to the $A$-basis $\{z_1,w_2\cdot z_1\}$ of $P^{-1}$,
$w_1$ (resp. $w_2$) acts on $P^{-1}$ as the matrix
\begin{equation}
\label{eq:p1actions}
\tilde{W}_{1,P^{-1}}=\left(\begin{array}{cc}a&b\\b&a\end{array}\right)
\quad \left(\mbox{resp. }\tilde{W}_{2,P^{-1}}=\left(\begin{array}{cc}0&1\\1&0\end{array}\right)
\right)
\end{equation}
where $a\in A^*$ and $b\in m_A$ satisfy $-s_1s_2=-s_2a+b$ and $s_1=-s_2b+a$. These two 
conditions are equivalent to the two conditions
\begin{eqnarray}
\label{eq:neeeeedthis}
b&=&s_2(a-s_1),\\
0&=&(a-s_1)(1-s_2^2).\nonumber
\end{eqnarray}
Since these are the only conditions
needed to ensure that $P^\bullet$ defines a quasi-lift
of $V_{\ell}^\bullet$ over $A$, 
we obtain the following: 
For all $s_1,s_2,a\in A^*$ with
$(a-s_1)(1-s_2^2)=0$, there is a two-term complex
\begin{equation}
\label{eq:ql}
Q^\bullet:\qquad  Q^{-1}=A\oplus A 
\xrightarrow{\left(\begin{array}{cc}-s_2&1\\1&-s_2\end{array}\right)}
A\oplus A=Q^0
\end{equation}
where 
$w_1$ (resp. $w_2$) acts on $Q^0$ as the matrix $W_{1,P^0}$ (resp. $W_{2,P^0}$) from 
$(\ref{eq:p0actions})$, and $w_1$ (resp. $w_2$) acts on $Q^{-1}$ as the matrix
$\tilde{W}_{1,P^{-1}}$ (resp. $\tilde{W}_{2,P^{-1}}$) from $(\ref{eq:p1actions})$ where
$b=s_2(a-s_1)$.
Moreover, for all choices of $s_1,s_2,a$ as above, 
$k\hat{\otimes}_A Q^\bullet$ is equal to the same complex $Z_\ell^\bullet$. 
By choosing suitable $k$-bases for the terms of $V_\ell^\bullet$, we see
that $V_\ell^\bullet=Z_\ell^\bullet$ in $C^-(k[[\Gamma]])$.
Thus for each choice of isomorphism
$\phi_Q:V_{\ell}^\bullet\to V_{\ell}^\bullet$ in $D^-(k[[\Gamma]])$, we obtain a quasi-lift 
$(Q^\bullet,\phi_Q)$ of $V_{\ell}^\bullet$ over $A$. 
Analyzing all isomorphisms in $\mathrm{Hom}_{D^-(k[[\Gamma]])}(V_\ell^\bullet,V_\ell^\bullet)$,
it follows that if $\phi_Q,\phi_Q':Z_\ell^\bullet=V_{\ell}^\bullet\to V_{\ell}^\bullet$ are isomorphisms
in $D^-(k[[\Gamma]])$, then $(Q^\bullet,\phi_Q)$ is isomorphic to $(Q^\bullet,\phi_Q')$ as quasi-lifts
of $V_\ell^\bullet$ over $A$.

Let $S_\ell=W[[t_1,t_2,t_3]]/(t_2t_3(t_3+2))$.
We obtain a two-term complex $U^\bullet$ in $C^-(S_\ell[[\Gamma]])$ from $Q^\bullet$
by replacing $A$ by $S_\ell$, 
$s_1$ by $1+t_1$, $a-s_1$ by $t_2$, $s_2$ by $1+t_3$ and $b$ by $(1+t_3)t_2$
in $(\ref{eq:ql})$, in $(\ref{eq:p0actions})$ and in $(\ref{eq:p1actions})$.
Let $\phi_U:k\hat{\otimes}_AU^\bullet
=Z_\ell^\bullet\to V_{\ell}^\bullet$ be a fixed isomorphism in $D^-(k[[\Gamma]])$.
Given a quasi-lift of $V_{\ell}^\bullet$ over $A$ which is isomorphic to 
$(Q^\bullet,\phi_Q)$ for $Q^\bullet$ as above,
it follows that the morphism
$\alpha:S_\ell \to A$ 
with $\alpha(t_1)=s_1-1$, $\alpha(t_2)=a-s_1$ and $\alpha(t_3)=s_2-1$ is a 
morphism in $\hat{\mathcal{C}}$ such that 
$(Q^\bullet,\phi_Q)$ is isomorphic to $(A\hat{\otimes}_{S_\ell,\alpha}U^\bullet,\phi_U)$ 
as quasi-lifts of $V_{\ell}^\bullet$ over $A$.
Because $\mathrm{max}(S_\ell)/(\mathrm{max}^2(S_\ell)+2S_\ell)$ has $k$-dimension $3$, it follows
from Lemma $\ref{lem:Ext1Vy}$ that $S_\ell$ is the versal deformation ring of $V_{\ell}^\bullet$.

For proflat quasi-lifts of $V_{\ell}^\bullet$ over a ring $A$ in $\hat{\mathcal{C}}$,
the only additional condition is $s_2^2=1$. It follows that the versal proflat deformation
ring of $V_{\ell}^\bullet$ is isomorphic to $W[[t_1,t_2,t_3]]/(t_3(t_3+2))$.
\end{proof}


\end{document}